%% file: template_arxiv.tex
\documentclass[10pt]{article}       
\usepackage{graphicx}
\usepackage{latexsym,amsfonts,amsmath,theorem,amssymb}
\usepackage{stmaryrd,array,tabularx,bbm}
\usepackage{pstricks,graphicx}
\usepackage{a4wide}
\usepackage{color}
\usepackage{caption}
\usepackage{theorem}
\usepackage{cleveref}
\usepackage{framed}
\usepackage{blindtext}
\usepackage{float}
\usepackage{todonotes}
\usepackage[caption=false]{subfig}
\usepackage{algorithm}
\usepackage{algorithmicx}
\usepackage[noend]{algpseudocode}
\usepackage{tikz}
\usetikzlibrary{arrows}
\usetikzlibrary{arrows.meta}
\usetikzlibrary{decorations.pathreplacing,angles,quotes}

\algrenewcommand\algorithmicrequire{\textbf{Input:}}
\algrenewcommand\algorithmicensure{\textbf{Output:}}
\newtheorem{theorem}{Theorem}
\newtheorem{lemma}{Lemma}
\newtheorem{proposition}{Proposition}
\newtheorem{assumption}{Assumption}
\newtheorem{definition}{Definition}

\newenvironment{proof}{\begin{trivlist}
    \item[\hskip\labelsep{\bf Proof.}]}{$\hfill\Box$\end{trivlist}}

{\theoremstyle{plain} \theorembodyfont{\rmfamily}
\newtheorem{example}{Example}
\newtheorem{remark}{Remark}}

\newcommand{\bsk}{{\boldsymbol{k}}}
\newcommand{\bsi}{{\boldsymbol{i}}}
\newcommand{\bsj}{{\boldsymbol{j}}}

 \newcommand{\hPhi}{{\bar{\Phi}}}

\renewcommand{\hat}{\widehat}

\newcommand{\revised}{\empty}
\newcommand{\vertiii}[1]{{\left\vert\kern-0.25ex\left\vert\kern-0.25ex\left\vert #1 
\right\vert\kern-0.25ex\right\vert\kern-0.25ex\right\vert}}

\graphicspath{{./figures/}}

\begin{document}

\title{Analysis of Stochastic Gradient Descent in Continuous Time\thanks{The author acknowledges support from the EPSRC grant EP/S026045/1 ``PET++: Improving Localisation, Diagnosis and Quantification in Clinical and Medical PET Imaging with Randomised Optimisation". The author is very grateful for insightful discussions with Claire Delplancke, Matthias Ehrhardt, and Carola-Bibiane Sch\"onlieb that contributed to this work. Furthermore, the author thanks Christian Etmann and Felipe Uribe for carefully reading and commenting on this manuscript. Finally, the author thanks the anonymous reviewers for their helpful and constructive reports.  }
}


\author{Jonas Latz     
}


\date{\footnotesize Department of Applied Mathematics and Theoretical Physics,
              University of Cambridge \\
              Wilberforce Road, Cambridge, CB3 0WA, United Kingdom \\
            \texttt{jl2160@cam.ac.uk}}

\maketitle

\begin{abstract}
Stochastic gradient descent is an optimisation method that combines classical gradient descent with random subsampling within the target functional. 
In this work, we introduce the stochastic gradient process as a conti\-nuous-time representation of stochastic gradient descent.
The stochastic gradient process is a dynamical system that is coupled with a continuous-time Markov process living on a finite state space. The dynamical system -- a  gradient flow -- represents the gradient descent part, the process on the finite state space represents the random subsampling. Processes of this type are, for instance, used to model clonal populations in fluctuating environments.
After introducing it, we study theoretical properties of the stochastic gradient process\revised{:} We show that it converges weakly to the gradient flow with respect to the full target function, as the learning rate approaches zero. \revised{We give conditions under which the stochastic gradient process with constant learning rate is} exponentially ergodic in the Wasserstein sense. \revised{Then we study the case, where} the learning rate goes to zero sufficiently slowly \revised{ and the single target functions are strongly convex}. In this case, the process converges weakly to the point mass concentrated in the global minimum of the full target function\revised{; indicating consistency of the method.}
We conclude \revised{after} a discussion of discretisation strategies for the stochastic gradient process and  numerical experiments.


\end{abstract}
\section{Introduction} \label{Sec_intro}
The training of models with \emph{big} data sets is a crucial task in modern machine learning and artificial intelligence. 
The training is usually phrased as an optimisation problem. Solving this problem with classical optimsation algorithms is usually infeasible.  Classical algorithms being \emph{gradient descent} or the \emph{(Gauss--)Newton method}; see  \cite{Nocedal2006}. 
Those methods require evaluations of the loss function with respect to the full \emph{big} data set in each iteration. This leads to an immense computational cost.

Stochastic optimisation algorithms that only consider a small fraction of the data set in each step have shown to cope well with this issue in practice; see, e.g., \cite{bottou-98x,Chambolle2018,robbins1951}. The stochasticity of the algorithms is typically induced by \emph{subsampling}. In subsampling the aforementioned small fraction of the data set is picked randomly in every iteration.
Aside from a higher efficiency, this randomness can have a second effect: The perturbation introduced by subsampling can allow to escape local extrema and saddle points. This is highly relevant for target functions in, e.g., deep learning, since those are often non-convex; see \cite{choromanska15,Vidal2017}.

Due to the randomness in the updates, the sequence of iterates of a stochastic optimisation algorithm forms a stochastic process; rather than a deterministic sequence. 
Stochastic properties of these processes have been hardly studied in the literature so far; see \cite{Benaim_book,Dieuleveut2017,Hu2019} for earlier studies. However, understanding these properties seems crucial for the construction of efficient stochastic optimisation methods.

In this work, we study the stochastic processes generated by the \emph{stochastic gradient descent (SGD)} algorithm. More precisely, the contributions of this work are:
\begin{enumerate}
\item We construct the \emph{stochastic gradient process (SGP)}, a continuous-time representation of SGD. We show that  SGP is a sensible continuum limit of SGD and discuss SGP from a biological viewpoint: a model of the same type is used to model growth and phenotypes of clonal populations living in randomly fluctuating environments.
\item We study the long-time behaviour of SGP: 
We give assumptions under which SGP \revised{with constant learning rate} has a unique stationary measure and converges to this measure in the Wasserstein distance at exponential rate. In this case, SGP is \emph{exponentially ergodic}. If the learning rate is decreasing to zero and additional assumptions hold, we will prove that SGP converges weakly to the Dirac measure concentrated in the global optimum.
\item We discuss discretisation strategies for SGP. Those will allow us to derive practical optimisation algorithms from SGP. We also discuss existing algorithms that can be retrieved in this way.
\item We illustrate \revised{and investigate} the stochastic gradient process and its stationary regime alongside with stochastic gradient descent in numerical experiments.
\end{enumerate}

This work is organised as follows: we introduce notation and background in the remainder of \S\ref{Sec_intro}. In \S\ref{Sec_Modelling}, we introduce the stochastic gradient process and justify our model choice. We study the long-time behaviour of SGP in \S\ref{Sec_longtime}. After discussing discretisation strategies for SGP in \S \ref{Sec_Discret}, we give numerical \revised{experiments} in \S\ref{Sec_NumIll} and conclude the work in \S \ref{Sec_conclu}. 

\subsection{Stochastic gradient descent} \label{Subse_SGD_Standard}
Let $(X, \|\cdot\|) := (\mathbb{R}^K, \|\cdot\|_2)$, let $\langle \cdot, \cdot \rangle$ be the associated inner product, and let $\mathcal{B}X := \mathcal{B}(X, \|\cdot\|)$ be the Borel $\sigma$-algebra on $X$. Functions defined throughout this work will be assumed to be measurable with respect to appropriate $\sigma$-algebras.  Let $\hPhi: X \rightarrow \mathbb{R}$ be some function attaining a global minimum in $X$.   We assume that $\hPhi$ is of the form
$$
\hPhi = \frac{1}{N}\sum_{i=1}^N \Phi_i.
$$
Here,  $N \in \mathbb{N} := \{1, 2, \ldots \}$, $N \geq 2$, and $\Phi_i: X \rightarrow \mathbb{R}$ is some continuously differentiable \revised{function}, for $i$ in the index set $I := \{1,\ldots,N\}$. 
In the following, we aim to solve the unconstrained optimisation problem 
\begin{equation} \label{eq_optprob}
 \theta^* \in   \mathrm{argmin}_{\theta \in X} \hPhi(\theta).
\end{equation}
Optimisation problems as given in \eqref{eq_optprob} frequently arise in data science and machine learning applications. Here $\hPhi$ represents the negative log-likelihood or loss function of some training data set $y$ with respect to some model.
An index $i \in I$ typically refers to a particular fraction $y_i$ of the data set $y$. Then, $\Phi_i(\theta)$ represents the negative log-likelihood of only this fraction $y_i$ given the model parameter $\theta \in X$ or  the associated loss, respectively.

For optimisation problems of this kind, we employ the \emph{stochastic gradient descent} algorithm, which was proposed by Robbins and Monro \cite{robbins1951}.
We sketch this method in Algorithm~\ref{alg_SGD}. In practice, it is implemented with an appropriate termination criterion.
\begin{algorithm}
\caption{Stochastic gradient descent}\label{alg_SGD} 
\begin{algorithmic}[1] \State initialise $\theta_0 \in X$ deterministically or randomly 
\State define non-increasing sequence $(\eta_k)_{k=1}^\infty \in (0,\infty)^{\mathbb{N}}$ \label{line_noninc}
\For{$k = 1, 2, \ldots$}
      \State sample $\bsi_k \sim \mathrm{Unif}(I)$ \label{line_Unifsampl}
      \State  $\theta_k \gets \theta_{k-1} - \eta_k \nabla \Phi_{\bsi_k}(\theta_{k-1})$ \label{line_update}
   \EndFor
   \State \textbf{return} $(\theta_k)_{k=0}^\infty$
\end{algorithmic}
\end{algorithm}

The elements of the sequence $(\eta_k)_{k=1}^\infty$ defined in Algorithm~\ref{alg_SGD} line~\ref{line_noninc} are called \emph{step sizes} or \emph{learning rates}.
SGD is typically understood as a gradient descent algorithm with inaccurate gradient evaluations: the inaccuracy arises since we randomly substitute $\hPhi$ by some $\Phi_i$.
If $\lim_{k \rightarrow \infty} \eta_k = 0$ \revised{sufficiently slowly}, one can  show convergence for convex target \revised{functions} $\hPhi$; see, e.g., \cite{Jentzen2018,Nemirovski2009}. 
Moreover, as opposed to descent methods with exact gradients, the inexact gradients can help the algorithm escaping local extrema and saddle points in non-convex problems; see, e.g., \cite{Hu2019}.

In this work, we consider gradient descent algorithms as time stepping discretisations of a certain gradient flow. 
The potential of this gradient flow is the respective target \revised{function} $\hPhi$, $\Phi_1, \ldots, \Phi_N$. Thus, we refer to these target \revised{functions} as \emph{potentials}. In SGD, the potentials of these gradient flows are randomly \emph{switched} after every time step. 

We now comment on the meaning of the learning rate $\eta_k$.
\begin{remark} \label{Rem_eta} 
In the gradient flow setting, the learning rate $\eta_k$ has two different interpretations/objectives:
\begin{enumerate}
\item[(i)] It represents the step size of the explicit Euler me\-thod that is used to discretise the underlying gradient flow.
\item[(ii)]It represents the length of the time interval in which the flow follows a certain potential $\Phi_{i}$ at the given iteration $k$, i.e. the time between t\revised{w}o switches of potentials.
\end{enumerate}
\end{remark}
Recently, several authors, e.g. \cite{Trillos2018,kuntz2019,Schillings2017}, have been studying the behaviour of algorithms and methods at their continuum limit; i.e. the limit as  $\eta_j \downarrow 0$. The advantage of such a study is that numerical aspects, e.g., arising from the time discretisation can be neglected. Also, a new spectrum of tools is available to analyse, understand, and interpret the continuous system.
If the continuous system is a good representation of the algorithm, we can sometimes use the results in the continuous setting to improve our understanding of the discrete setting.

Under some assumptions,  a \emph{diffusion process} is a good choice for a continuous\revised{-}time model of SGD. Diffusion processes, such as Langevin dynamics, are traditionally used in statistical physics to represent the motion of particles; see, e.g., \S8 in \cite{Schwabl2006}.
\subsection{Diffusions and piecewise-deterministic Markov processes}
Under assumptions discussed in \cite{Hu2019}\revised{\cite{Li2019}}, one can show that the sequence of iterates of the SGD algorithm\revised{, with, say, constant $(\eta_k)_{k=1}^\infty \equiv \eta$,} can  be approximated by a stochastic differential equation of the following form:
\begin{align} 
  {\mathrm{d}\tilde{\theta}(t)} &= -\nabla\hPhi\big(\tilde{\theta}(t)\big){\mathrm{d}t} + \revised{\sqrt{\eta}} \Sigma\big(\tilde{\theta}(t)\big)^{\revised{1/2}} \mathrm{d}{W}(t) \quad (t > 0), \notag  \\
    \tilde{\theta}(0) &= \theta_0. \label{Eq_Diffusion_process}
\end{align}
Here, $\Sigma(\theta): X \rightarrow X$ is \revised{symmetric, }positive semi-definite for $\theta \in X$ and $W:[0, \infty) \rightarrow X$ is a $K$-dimensional Brownian motion.  \revised{`Can be approximated' means that as $\eta$ goes to zero, the approximation of SGD via such a diffusion process is precise in a weak sense. 
In the following remark, we give a (rather coarse) intuitive explanation, how this diffusion process could be derived using the Central Limit Theorem  and discretisation schemes for stochastic differential equations.
\begin{remark}  Let $\eta \approx 0$. Then, for some $k \in \mathbb{N}$, we have
\begin{align*}
\theta_k = \theta_{k-1} - \eta\nabla \Phi_{\bsi_k}(\theta_{k-1}) 
&\approx \theta_0 - \eta \sum_{\ell=1}^{k} \nabla \Phi_{\bsi_\ell}(\theta_{0}) \\ &= \theta_0 - \eta k  \sum_{\ell=1}^{k} \frac{\nabla \Phi_{\bsi_\ell}(\theta_{0}) }{k}
\end{align*}
The term $\sum_{\ell=1}^{k} \frac{\nabla \Phi_{\bsi_\ell}(\theta_{0}) }{k}$ is now the sample mean of a finite sample of independent and identically distributed (i.i.d.) random variables with finite variance. Hence, by the Central Limit Theorem, 
$$\sum_{\ell=1}^{k} \frac{\nabla \Phi_{\bsi_\ell}(\theta_{0}) }{k} \approx \nabla\hPhi(\theta_0) + \frac{\gamma_0}{\sqrt{k}}, $$ where $\gamma_0 \sim \mathrm{N}(0, \Sigma(\theta_0))$ and the  covariance matrix is given by $$\Sigma(\theta_0) := \frac{1}{N}\sum_{i \in I}(\nabla\Phi_i(\theta_0)-\hPhi(\theta_0))(\nabla\Phi_i(\theta_0)-\hPhi(\theta_0))^T.$$ Then, we have 
$$
\theta_k \approx \theta_{0} - \eta k \nabla\hPhi(\theta_0) - \sqrt{\eta k} \sqrt{\eta}\gamma_0,
$$
which is the first step of an Euler--Maruyama discretisation of the diffusion process in \eqref{Eq_Diffusion_process} with step size $\eta k$. See, e.g., \cite{lord_powell_shardlow_2014} for details on discretisation strategies for stochastic differential equations.
\end{remark}}

Th\revised{e} diffusion view \eqref{Eq_Diffusion_process} of SGD has been discussed by Li et al. \cite{Weinan2,Weinan3,Weinan1} and Mandt et al. \cite{Mandt2016,Mandt2017}. Moreover, it forms the basis of the Stochastic Gradient Langevin MCMC algorithm \cite{Mandt2017,Welling2011}. \revised{A diffusive continuous-time version of stochastic gradient descent also arises when the underlying target functional itself contains a continuous data stream; see \cite{Sirignano2017SIFIN,Sirignano2019}; this however is not the focus of the present work.}

Unfortunately, the process of slowly switching between \revised{a finite number of} potentials in the pre-asymptotic phase \revised{of SGD} is not represented in the diffusion. 
Indeed, the diffusion represents an infinite amount of switches within any strictly positive time horizon. In SGD this is only the case as $\eta_k\downarrow 0$; see \cite{Brosse2018}.
The pre-asymptotic phase\revised{, however, is vital for the robustness of the algorithm and its computational efficiency. Moreover}, the SGD algorithm is sometimes applied with a constant learning rate; see \cite{Chee2018}. Here, the regime $\eta_k\downarrow 0$ is never reached.
\revised{Finally, one motivation for this article has been the creation of new stochastic optimisation algorithms. 
Here, the switching between a finite number of potentials/data sets is a crucial element to reduce computational cost and memory complexity. 
Replacing the subsampling by a full sampling and adding Gaussian noise is not viable in large data applications.}

In this work, we aim to propose a continuous-time model of SGD that captures the switching of the \revised{finite number of} potentials. To this end we separate the two different learning rate objects: the gradient flow discretisation and the waiting time between two switches of potentials; see Remark~\ref{Rem_eta} (i) and (ii) respectively. We proceed as follows:
\begin{enumerate}
\item We let the discretisation step width go to zero and thus obtain a gradient flow with respect to some potential $\Phi_i$.
\item We randomly replace $\Phi_i$ by another potential $\Phi_j$ after some strictly positive waiting time.
\end{enumerate}
Hence, we take the continuum limit \emph{only} in the discretisation of the gradient flows, but not in the switching of potentials.
\revised{This gives us a continuous\revised{-}time dynamic in which the randomness is not introduced by a diffusion, but by an evolution according to a potential that is randomly chosen from a finite set. This non-diffusive approach should give a better representation of the pre-asymptotic phase. Moreover, since we do not require the full potential in this dynamical system, we obtain a representation that is immediately relevant for the construction of new computational methods.}

We will model the waiting times $T$ between two switches as a random variable following a \emph{failure distribution}, i.e. $T$ has survival function \begin{equation} \label{eq_failuredist}
\mathbb{P}(T \geq t) := \mathbf{1}[t < 0] + \exp\left(-\int_{0}^t \nu(u+t_0) \mathrm{d}u \right)\mathbf{1}[t \geq 0]
\end{equation}
where $t \in \mathbb{R}$, $t_0 \geq 0$ is the current time, $\nu: [0, \infty) \rightarrow (0, \infty)$ is a \emph{hazard function} that depends on time, and $\mathbf{1}[\cdot]$ represents the indicator function: $\mathbf{1}[\text{true}] := 1$ and $\mathbf{1}[\text{false}]:=0$. We denote $\mathbb{P}(T \in \cdot) =: \pi_{\rm wt}(\cdot | t_0)$.
Note that when $\nu$ is constant, $T$ is \emph{exponentially distributed}.

Then, we obtain a so-called \emph{Markov switching process}; see, e.g. \cite{Bakhtin2012,benaim2012:quant,benaim2015:qual,Cloez2015,Yin2010}. Markov switching processes are a subclass of \emph{piecewise deterministic Markov processes} (PDMPs). 
PDMPs were first introduced by Davis \cite{Davis84} as `a general class of non-diffusion stochastic models'; see also \cite{davis1993}.
They play a crucial role in the modelling of biological, economic, technical, and physical systems; e.g., as a model for internet traffic \cite{Graham2011} or in risk analysis \cite{Kritzer2019}. See also \S\ref{Subs_Biolo}, where we discuss a particular biological system that is modelled by a PDMP.
Furthermore, PDMPs have recently gained attention in the Markov chain Monte Carlo literature as efficient way of sampling from inaccessible probability distributions; see, e.g., \cite{bierkens2019,fearnhead2018,power2019}.

\section{From discrete to continuous} \label{Sec_Modelling}
In the following, we give a detailed description of the two PDMPs that will be discussed throughout this article: One PDMP will represent SGD with constant learning rate, the other PDMP models SGD with decreasing learning rate. Then, we will argue, why we believe that these PDMPs give an accurate continuous-time representation of the associated SGD algorithms.
Finally, we give a biological interpretation of the PDMPs discussed in this section.
\subsection{Definition and well-definedness}
Let $(\Omega, \mathcal{A}, \mathbb{P})$ be a probability space on which all random variables in this work are defined.
We now define two \emph{continuous-time Markov processes (CTMPs)} on $I$\revised{$:= \{1,\ldots,N\}$} that will model the switching of the data sets in our PDMPs. For details on continuous-time Markov processes on finite state spaces, we refer to \cite{Anderson1991}. 
We start with the constant learning rate.
Let $\lambda > 0$ be a positive constant and let $\bsi: \Omega \times [0, \infty) \rightarrow I$ be the CTMP on $I$ with transition rate matrix 
\begin{equation} \label{eq_transition_Mat}
    A := \begin{pmatrix}\lambda & \cdots & \lambda \\ \vdots & \ddots & \vdots \\ \lambda & \cdots & \lambda \end{pmatrix} - N\lambda \cdot  \mathrm{Id}_I 
\end{equation}
and with initial distribution $\bsi(0) \sim \mathrm{Unif}(I). $ Here, $\mathrm{Id}_I$ is the identity matrix in $\mathbb{R}^{N \times N}$.
Let $M_t: I \times 2^I \rightarrow[0,1]$ be the Markov kernel representing the semigroup of $(\bsi(t))_{t \geq 0}$, i.e. $$M_t(\cdot | i_0) := \mathbb{P}(\bsi(t) \in \cdot| \bsi(0) = i_0) \qquad (i_0 \in I, t \geq 0).$$
This Markov kernel can be represented analytically by solving the associated Kolmogorov forward equation. \revised{We do this in Lemma~\ref{Lemma_Appendix_M} in Appendix~\ref{Appendix_CTMPs} and show that}
\begin{align} \label{eq_transitionKernel}
    M_t(\{i\}|i_0) = \frac{1-\exp(-\lambda N t)}{N} &+ \exp(-\lambda N t) \mathbf{1}[i = i_0],
\end{align}
where $i, i_0 \in I, t \geq 0$.
Moreover, note that the waiting time between two jumps of the process $(\bsi(t))_{t\geq 0}$ is given by an exponential distribution with rate $(N-1)\lambda$, i.e. $\pi_{\rm wt}(\cdot | t_0) = \mathrm{Exp}((N-1)\lambda)$.
The CTMP $(\bsi(t))_{t\geq 0}$ will represent the switching among potentials in the SGD algorithm with \emph{constant learning rate}.

Now, we move on to the case of a decreasing learning rate. Let $\mu: [0, \infty) \rightarrow (0,\infty)$ be a non-decreasing, positive, and \revised{continuously differentiable} function, with $\mu(t) \rightarrow \infty,$ as $t \rightarrow \infty$. \revised{We assume furthermore that  for any $\overline{t} > 0$ there is some constant $C_{\overline{t} }> 0$ : \begin{equation} \label{Eq_bounded_deriv}
\left\lvert\frac{\partial\mu}{\partial t}(t)\right\rvert \leq C_{\overline{t} } \mu(t) \qquad (t \in [0,\overline{t}]).
\end{equation}}

We define $\bsj: \Omega \times [0, \infty) \rightarrow I$ to be the inhomogeneous CTMP with time-dependent transition rate matrix $B: [0, \infty) \rightarrow \mathbb{R}^{N \times N}$ given by
\begin{align}
 B(t) := \begin{pmatrix}\mu(t) & \cdots & \mu(t) \\ \vdots & \ddots & \vdots \\ \mu(t) & \cdots & \mu(t) \end{pmatrix} - N\mu(t) \cdot  \mathrm{Id}_I  \qquad (t \geq 0).
\end{align}
Again, we assume that the initial distribution $\bsj(0) \sim \mathrm{Unif}(I)$.  Equivalently to \eqref{eq_transitionKernel}, we can compute the associated Markov transition kernel in this setting. First note that since $(\bsj(t))_{t \geq 0}$ is not homogeneous in time, it is not sufficient to construct the Markov kernel with respect to the state of the Markov process at time $t_0 = 0$. Indeed, we get a kernel of type
$$M_{t|t_0}'(\cdot | j_0) := \mathbb{P}(\bsj(t) \in \cdot| \bsj(t_0) = j_0),$$
where  $j_0 \in I$ and $ t \geq t_0 \geq 0.$
This kernel is given by
\begin{align} 
M_{t|t_0}'(\{j\}|j_0) = &\frac{1-\exp\left(-N \int_{t_0}^t \mu(u) \mathrm{d}u\right)}{N} + \exp\left(- N \int_{t_0}^t \mu(u) \mathrm{d}u\right) \mathbf{1}[j = j_0], \label{eq_transitionKernel_decreasing_case}
\end{align}
where $j, j_0 \in I$ and $t \geq t_0 \geq 0$\revised{; see again Lemma~\ref{Lemma_Appendix_M} in Appendix~\ref{Appendix_CTMPs}.}
In this case, the waiting time at time $t_0 \geq 0$ between two jumps is distributed according to the failure distribution $\pi_{\rm wt}$ in \eqref{eq_failuredist}, with $\nu \equiv (N-1)\mu$.
The CTMP $(\bsj(t))_{t\geq 0}$ represents the potential switching when SGD has \emph{decreasing learning rates}.

Based on these Markov jump processes, we can now define the stochastic gradient processes that will act as continuous\revised{-}time version of SGD as defined in Algorithm~\ref{alg_SGD}.
\begin{definition}[SGP]  \label{Def_SGD_MP}
Let $\theta_0, \xi_0 \in X$. We define  
\begin{enumerate}
\item[(i)] the \emph{stochastic gradient process  with constant learning rate (SGPC)} as a solution of the initial value problem
\begin{align} \label{eq_SGPC}
    \frac{\mathrm{d}\theta(t)}{\mathrm{d}t} &= - \nabla \Phi_{\bsi(t)}(\theta(t)), \qquad \theta(0) = \theta_0,
\end{align} 
\item[(ii)] the \emph{stochastic gradient process with decreasing learning rate (SGPD)} as a solution of the initial value problem
\begin{align} \label{eq_SGPD}
    \frac{\mathrm{d}\xi(t)}{\mathrm{d}t} &= - \nabla \Phi_{\bsj(t)}(\xi(t)), \qquad \xi(0) = \xi_0.
\end{align}
\end{enumerate}
Also, we use the denomination \emph{stochastic gradient process} (SGP) when referring to (i) and (ii) at the same time. 
\end{definition}
We illustrate the processes $(\bsi(t))_{t \geq 0}$ and $(\theta(t))_{t \geq 0}$ in Figure~\ref{Fig_pdmp_cartoon}. We observe that SGP constructs a piecewise smooth path that is smooth between jumps of the underlying CTMP.

\input{figure_PDMP}

In order to show that the dynamics in Definition~\ref{Def_SGD_MP} are well-defined, we require regularity assumptions on the potentials $(\Phi_i)_{i \in I}$. 
After stating those, we immediately move on with proving well-definedness in Proposition~\ref{Prop_PicLind}.

\begin{assumption} \label{Ass_Cinf} For any $i \in I$, let $\Phi_i: X \rightarrow \mathbb{R}$ be  continuously differentiable, i.e. $\Phi_i \in C^{1}(X; \mathbb{R})$, and let $\nabla \Phi_i$ be \revised{locally} Lipschitz continuous.
\end{assumption}

\begin{proposition}\label{Prop_PicLind}
Let Assumption~\ref{Ass_Cinf} hold. Then, the initial value problems \eqref{eq_SGPC} and \eqref{eq_SGPD} have a unique solution for $\mathbb{P}$-almost any realisation of the CTMPs $(\bsi(t))_{t \geq 0}$ and $(\bsj(t))_{t \geq 0}$, and for any initial values $\theta_0, \xi_0 \in X$. Moreover, the sample paths $t \mapsto \theta(t)$ and  $t \mapsto \xi(t)$ are $\mathbb{P}$-almost surely in $C^0([0, \infty); X).$ 
\end{proposition}
\begin{proof} We \revised{first} discuss the process $(\theta(t))_{t \geq 0}$. 
Let $T_0 = 0$ and $T_1,T_2,\ldots$ be the jump times of $(\bsi(t))_{t \geq 0}$. Let $k \in \mathbb{N}$.
Note that the increments $T_k - T_{k-1} \sim \mathrm{Exp}((N-1)\lambda)$. Hence, $\mathbb{P}(T_k - T_{k-1} > 0) = 1$.
By Assumption~\ref{Ass_Cinf} the $(\Phi_i)_{i=1}^N$ are locally Lipschitz continuous. Hence, the process $(\theta(t))_{t \geq 0}$ can be defined iteratively on the intervals
\begin{align*}
\frac{\mathrm{d}\theta(t)}{\mathrm{d}t} &= - \nabla \Phi_{\bsi(t)}(\theta(t))  &(t \in [T_{k-1}, T_k)), \\ \theta(T_{(k-1)}) &=\theta(T_{(k-1)}-) &(k \in \mathbb{N}),
\end{align*}
where $f(x-) := \lim_{x' \uparrow x}f(x')$  and $T_0-:= 0$.
Iterative application of the Picard--Lindel\"of Theorem for $k \in \mathbb{N}$ gives unique existence of the trajectory. Picard--Lindel\"of can be applied, since $\nabla \Phi_i$ is locally Lipschitz continuous for any $i \in I$ by Assumption~\ref{Ass_Cinf}.

\revised{The proof for $(\xi(t))_{t \geq 0}$ is partially analogous. Importantly, we now need to make sure that $$\mathbb{P}\left(\lim_{k\rightarrow \infty} T_k = \infty\right) = 1.$$ Otherwise, $(\bsj(t))_{t \geq 0}$ would only be well-defined up to a possibly finite explosion time $T_\infty:=\lim_{k\rightarrow \infty} T_k <\infty$. Under our assumptions, $(\bsj(t))_{t \geq 0}$ is indeed  `non-explo\-sive', we prove this in Lemma~\ref{Lemma_non-explosive} in Appendix~\ref{Appendix_CTMPs}. Moreover, let $k \in \mathbb{N}$. Then, we have  $$\mathbb{P}(T_k - t_{k-1} > 0) = \pi_{\rm wt}((0,\infty)|t_{k-1})= 1,$$ for any $t_{k-1} \geq 0$. This is implied by the continuous differentiability of $\mu$. Thus, we also have $$\mathbb{P}(T_k - T_{k-1} > 0) =1.$$  Then, as for $(\theta(t))_{t \geq 0}$ we can employ again Picard--Lindel\"of iteratively to show the $\mathbb{P}$-a.s. well-definedness of $(\xi(t))_{t \geq 0}$.} 
\end{proof}

\subsection{Choice of model} \label{Subsec_Choice_of_model}
In this section, we reason why the dynamical systems in Definition~\ref{Def_SGD_MP} are  sensible continuous-time models for SGD given in Algorithm~\ref{alg_SGD} with constant, resp.  decreasing learning rate. 

\paragraph{Gradient flow.} The update in line~\ref{line_update} of Algorithm~\ref{alg_SGD} is an explicit Euler update of the gradient flow with respect to the potential $\Phi_i$, for some $i \in I$. 
In this model, we replace this \revised{discretised} gradient flow with the \revised{continuous} dynamic. Hence, we replace
$$
\theta \leftarrow \theta - \eta \nabla \Phi_i(\theta) \qquad  \text{ by } \qquad  \frac{\mathrm{d}\theta(t)}{\mathrm{d}t} = - \nabla \Phi_{i}(\theta(t)).
$$

\paragraph{Uniform sampling.} We aim to accurately represent the uniform sampling from the index set $I$, given in line~\ref{line_Unifsampl} of the algorithm. Indeed, at each point in time $t \in [0, \infty)$, we can show that both $\bsi(t) \sim \mathrm{Unif}(I)$ and $\bsj(t) \sim \mathrm{Unif}(I)$.

\begin{proposition} We have
$\mathbb{P}(\bsi(t) \in \cdot)=\mathbb{P}(\bsj(t) \in \cdot)= \mathrm{Unif}(I)$ for any $t \geq 0$.
\end{proposition}
\begin{proof} To prove this proposition, we need to show that $\mathrm{Unif}(I)$ is stationary with respect to the Markov transition kernels $M_t$ and $M'_{t|t_0}$ given in \eqref{eq_transitionKernel} and \eqref{eq_transitionKernel_decreasing_case}, respectively. In particular, we need to show that
$$
\mathrm{Unif}(I) M_t(\{i\}| \cdot ) = \mathrm{Unif}(I)M'_{t|t_0}(\{i\}|\cdot) =  \mathrm{Unif}(I)(\{i\}),$$ for $i \in I$ and $0 \leq t_0 \leq t.$
We show only the decreasing learning rate case, the proof for the constant learning rate proceeds analogously. 
A calculation gives:
\begin{align*}
\mathrm{Unif}(I)M_{t|t_0}'(\{i\}|\cdot) &= \int_{I} M_{t|t_0}'(\{i\}|i_0)  \mathrm{Unif}(I)(\mathrm{d}i_0) \\ &= \frac{1}{N} \exp\left(- N \int_{t_0}^t \mu(u) \mathrm{d}u\right) + \frac{1}{N} \sum_{i_0=1}^{N}\frac{1-\exp\left(-N \int_{t_0}^t \mu(u) \mathrm{d}u\right)}{N} \\ &= \frac{1}{N} = \mathrm{Unif}(I)(\{i\}),
\end{align*}
for any $i \in I$ and $0 \leq t_0 \leq t$. 
\end{proof}
Hence, the CTMPs $(\bsi(t))_{t\geq 0}, (\bsj(t))_{t \geq 0}$ indeed represent the uniform sampling among the data set indices $i \in I$.

\paragraph{Markov property.}
The trajectory $(\theta_k)_{k=0}^\infty$ generated by Algorithm~\ref{alg_SGD} satisfies the Markov property, i.e. the distribution of the current state given information about previous states is equal to the distribution of the current state given only information about the most recent of the previous states.
By the particular structure we chose for the continuous-time processes $(\theta(t), \bsi(t))_{t > 0}$ and $(\xi(t), \bsj(t))_{t > 0}$, we indeed retain the Markov property. 
\begin{proposition}
$(\theta(t), \bsi(t))_{t \geq 0}$ and $(\xi(t), \bsj(t))_{t \geq 0}$ are \\ Markov processes.
\end{proposition}
\begin{proof}  This follows from the particular choice of waiting time distribution, see e.g. the discussion in \S3 of \cite{Davis84}.
\end{proof}

Choosing random waiting times between switches allows us to analyse SGD as a PDMP. However, this choice comes at some cost. In Algorithm~\ref{alg_SGD}, the waiting times are all deterministic; a feature we, thus, do not represent in SGP. \revised{We briefly discuss a continuous-time version of SGD with deterministic waiting times in Remark~\ref{Rem_nonMarkovian} as a potential extension of the SGP framework, but do not consider it otherwise in this work.}
In the next two steps, we will\revised{, thus,} explain how we connect the deterministic waiting times in SGD and the random waiting times in SGP.

\paragraph{Constant learning rate.} We have defined $(\theta(t))_{t \geq 0}$ as a continuous\revised{-}time representation of the trajectory returned by Algorithm~\ref{alg_SGD} with a constant learning rate $\eta_k \equiv \eta$.
The hazard function of the waiting time distribution of $(\bsi(t))_{t \geq 0}$ is just constant $\nu \equiv (N-1)\lambda$. The waiting time $T$ is the time $(\bsi(t))_{t \geq 0}$ remains in a certain state. 
Note that the hazard function satisfies
\begin{equation*}
\nu(u) = \lim_{d \rightarrow 0} \frac{\mathbb{P}(u \leq T \leq u + d| T \geq u)}{d},
\end{equation*}
where $T$ is a waiting time; see, e.g., \S 21 in \cite{davis1993}. Hence, the hazard function describes the rate of events happening  at time $u \geq 0$.
In SGD with constant learning rate, the waiting time is constant $\eta$. Hence, the number of data switches in a unit interval is $1/\eta$.
Hence, we mimic this behaviour by choosing $\lambda$ in the matrix $A$ such that it satisfies $(N-1)\lambda = 1/\eta$. Indeed, we set $\lambda := 1/( (N-1)\eta)$.

\paragraph{Decreasing learning rate.}
Let now $(\eta_k)_{k=1}^\infty \in (0, \infty)^{\mathbb{N}}$ be a non-increasing sequence of learning rates, with $\lim_{k \rightarrow \infty} \eta_k = 0$. Moreover, we assume that $\sum_{k=1}^\infty \eta_k = \infty$.  Similarly to the last paragraph, we now try to find a rate function $(\mu(t))_{t \geq 0}$ such that the PDMP $(\xi(t))_{t \geq 0}$ represents the SGD algorithm with the sequence of learning rates $(\eta_k)_{k=1}^\infty$. 
To go from discrete time to continuous time, we need to define a function $H$ that interpolates the sequence of learning rates $\eta_k$, i.e. $H:[0,\infty) \rightarrow (0, \infty)$ is a non-increasing, \revised{continuously differentiable} function, such that 
$$
H(0) = \eta_1, \quad H\left(t_k\right) = \eta_{k+1}, \quad t_k := \sum_{\ell = 1}^k{\eta}_\ell \quad (k \in \mathbb{N}),
$$
where the $t_k$ are chosen like this, since the $\eta_k$ themselves represent the time stepsizes in the sequence of learning rates.
$H$ could for instance be chosen as a \revised{sufficiently smooth} interpolant between the $\eta_k$. Equivalently to the case of the constant learning rate, we now  argue via the hazard function of the waiting time distribution $\nu(t) := (N-1) \mu(t)$ $(t \geq 0)$ that $\mu(t) := 1/((N-1)H(t))$ \revised{$(t \geq 0)$} is a reasonable choice for the waiting time distribution.

\paragraph{Approximation of the exact gradient flow.} We now consider SGD, i.e. Algorithm~\ref{alg_SGD}. If the learning rate $\eta \downarrow 0$, we discretise the gradient flow precisely. Moreover, the waiting time between two data switches goes to zero. Hence, intuitively we switch the data set infinitely often in any finite time interval. By the Law of Large Numbers, we should then anticipate that the limiting process behaves like the \emph{full gradient flow}
\begin{equation} \label{eq:zeta_ode}
 \frac{\mathrm{d}\zeta(t)}{\mathrm{d}t} = - \nabla \hPhi(\zeta(t)),   
 \end{equation}
with initial value $\zeta(0) = \zeta_0  :=\theta_0$ as chosen in SGPC and $\hPhi := \sum_{i=1}^N \Phi_i/N$ being the full potential.

\revised{This behaviour can also be seen in the diffusion approximation to SGD \eqref{Eq_Diffusion_process}, where the stochastic part disappears as $\eta \downarrow 0$.}
 
 \revised{So we should now show that this is also true for SGPC}.  Indeed, we will give assumptions under which the \revised{SGPC} $(\theta(t))_{t \geq 0}$ converges weakly to $(\zeta(t))_{t \geq 0}$\revised{, as $\eta \downarrow 0$.}
\emph{Weak convergence} of  $(\theta(t))_{t \geq 0}$ to $(\zeta(t))_{t \geq 0}$ means that
\begin{equation} \label{eq_weak_conv}
\int_{\revised{\Omega}} F\left((\theta(t))_{t \geq 0}\right) \mathrm{d}\mathbb{P} \longrightarrow \int_{\revised{\Omega}} F\left( (\zeta(t))_{t \geq 0} \right) \mathrm{d}\mathbb{P} \qquad \revised{(\eta \downarrow 0)},
\end{equation}
for any bounded, continuous function $F$ mapping from $C^0([0, \infty); X)$ to $\mathbb{R}.$
Here, $C^0([0, \infty); X)$ is equipped with the supremum norm $\|f\|_\infty := \sup_{t \in [0, \infty)}\|f(t)\|$. We denote weak convergence by $(\theta(t))_{t \geq 0} \Rightarrow (\zeta(t))_{t \geq 0}$.

To show weak convergence, we need some stronger smoothness assumption concerning the potentials $\Phi_i$.
We denote the Hessian of $\Phi_i$ by $\mathrm{H} \Phi_i$ for $i \in I$.
\begin{assumption} \label{Ass_continuity} For any $i \in I$,
let $\Phi_i \in C^2(X; \mathbb{R})$ and let $\nabla \Phi_i, \mathrm{H}\Phi_i$ be continuous.
\end{assumption}
Please note that Assumption~\ref{Ass_Cinf} is already implied by Assumption~\ref{Ass_continuity}. 

\begin{theorem} \label{thm:ODE_limit}
Let $\theta_0= \zeta_0$ and let Assumption~\ref{Ass_continuity} hold, then \revised{ the stochastic gradient process $(\theta(t))_{t \geq 0}$ converges weakly to the full gradient flow $(\zeta(t))_{t \geq 0}$, as the learning rate $\eta \downarrow 0$; i.e. }
$
(\theta(t))_{t \geq 0} \Rightarrow (\zeta(t))_{t \geq 0}$, as $\eta \downarrow 0$.
\end{theorem}
We prove Theorem~\ref{thm:ODE_limit} rigorously in \S\ref{Appendix_Theorem_ODE_limit}. 
We illustrate the shown result in Figure~\ref{Fig_odelim}, where we can see that indeed as $\eta$ decreases, the processes converge to the \revised{full} gradient flow.

\begin{figure*}
\centering
\includegraphics[scale=0.85]{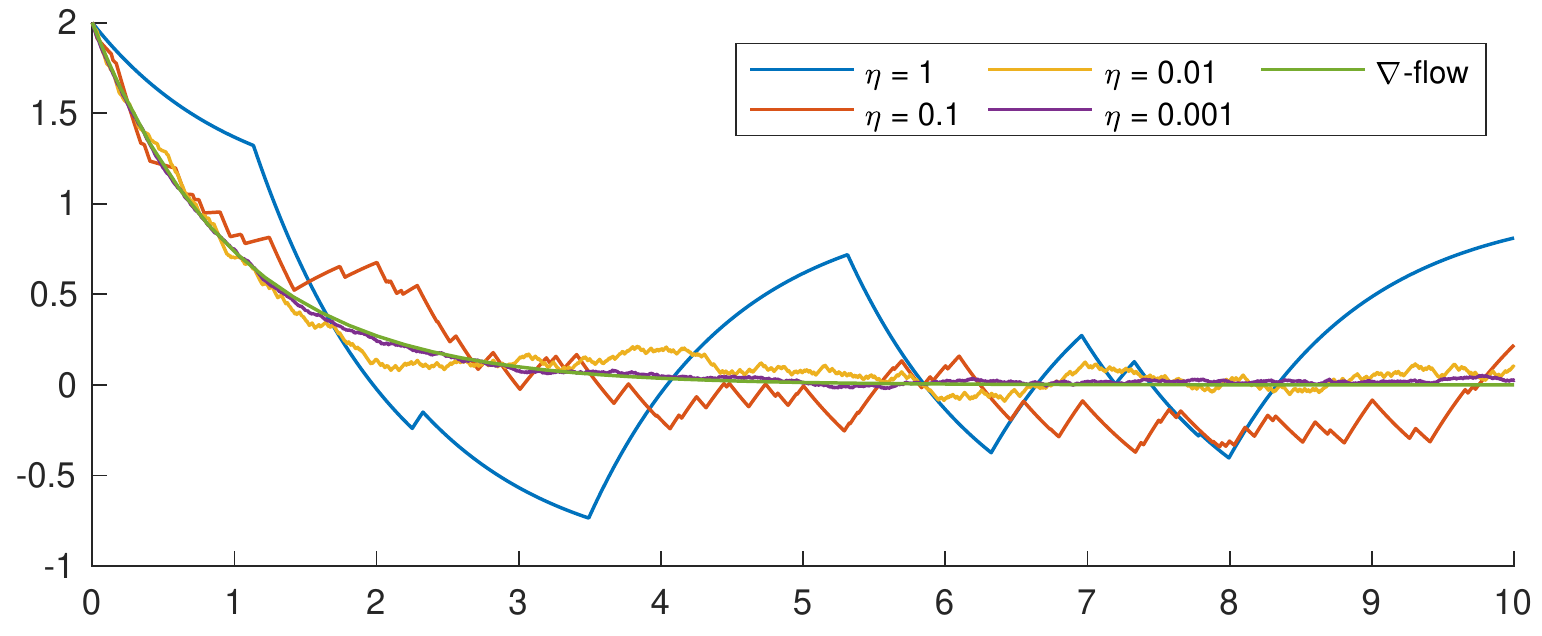}
 \caption{Exemplary realisations of SGPC for potentials $\Phi_1(\theta) := (\theta-1)^2/2$ and $\Phi_2(\theta) := (\theta+1)^2/2$ and learning rates $\eta \in \{0.001, 0.01, 0.1, 1 \}$ and a plot of the \revised{full} gradient flow corresponding to $\hPhi := \Phi_1/2 + \Phi_2/2$. The latter has $0$ as a stationary point. The ODEs are solved with \texttt{ode45} in \textsc{Matlab} - an explicit high-order Runge-Kutta method with adaptive discretisation step size.}
    \label{Fig_odelim}
\end{figure*}

Following our reasoning above, we assert that SGPC, resp. SGPD, are suitable continuous-time representations of SGD with constant, resp. decreasing, learning rate. 

\subsection{Proof of Theorem~\ref{thm:ODE_limit}} \label{Appendix_Theorem_ODE_limit}
We prove Theorem~\ref{thm:ODE_limit} using the \emph{perturbed test function theory}. In particular, we apply a result from \cite{Kushner1984} that we summarise below. We note that a similar technique is used to derive the \emph{Infinite Swapping Markov Chain Monte Carlo technique}; see \cite{Dupuis2012} for details from the statistical mechanics viewpoint and \cite{LMCNT20} for the discrete-time MCMC viewpoint. In the following, we adapt the notation of \cite{Kushner1984}.

Let $(\xi^{\varepsilon}(t))_{t \geq 0}$ be a right-continuous stochastic process on $Y \subseteq \mathbb{R}^{L}$ that depends on $\varepsilon > 0$. Moreover, let $G: X \times Y \rightarrow X$ and $\bar{G}: X \rightarrow X$ be vector fields on $X$. Moreover, let $x_0, x_0^{\varepsilon} \in X$.
Let now $(x^{\varepsilon}(t))_{t \geq 0}$ be the stochastic process generated by
$$
\frac{\mathrm{d}x^{\varepsilon}(t)}{\mathrm{d}t} = G(x^{\varepsilon}(t), \xi^{\varepsilon}(t)), \qquad x^{\varepsilon}(0) = x_0^{\varepsilon}.
$$
Moreover, let $(x(t))_{t \geq 0}$ solve the following ODE:
$$
\frac{\mathrm{d}x(t)}{\mathrm{d}t} = \bar{G}(x(t)), \qquad x(0) = x_0.
$$
We will now give assumptions under which $(x^{\varepsilon}(t))_{t \geq 0} \Rightarrow (x(t))_{t \geq 0}$ as $\varepsilon \downarrow 0$.

\begin{assumption} \label{Assum_kushner} We consider the following three assumptions:
\begin{itemize}
    \item[(i)] Let $G$ and $\nabla_x G$ be continuous and bounded on $X' \times Y$, where $X' \subseteq X$ is bounded,
    \item[(ii)] let $\bar{G}: X \rightarrow X$ be continuously differentiable and let for any $0 \leq \underline{t} < \overline{t} < \infty$ and $ x \in X$:
$$
\int_{\underline{t}}^{\overline{t}} \mathbb{E}[G(x, \xi^{\varepsilon}(s)) - \bar{G}(x)| \{\xi^{\varepsilon}(s') : s' \leq \underline{t} \}] \mathrm{d}s \rightarrow 0,
$$
in probability, as $\varepsilon \downarrow 0$, and
\item[(iii)] let $(\xi^{\varepsilon}(t))_{t \geq 0}$ be tight with respect to $\varepsilon$.
\end{itemize}
\end{assumption}
The associated result reads then:
\begin{theorem}[Kushner, 1984] \label{Thm_kushner} Let Assumption~\ref{Assum_kushner} (i)-(iii) hold. Moreover, let $x_0^{\varepsilon} \Rightarrow x_0$, as $\varepsilon \downarrow 0$.  Then, $(x^{\varepsilon}(t))_{t \geq 0} \Rightarrow (x(t))_{t \geq 0}$, as $\varepsilon \downarrow 0$.
\end{theorem}
\begin{proof} The proof uses the perturbed test function me\-thod; see \cite[Theorem 4.1]{Kushner1984}.
\end{proof}

To prove Theorem~\ref{thm:ODE_limit}, we now show that Assumption~\ref{Assum_kushner} (i)-(iii) hold for SGPC. Then, Theorem~\ref{Thm_kushner} will imply weak convergence.
\begin{proof}[{Proof of Theorem \ref{thm:ODE_limit}}]
We commence by transferring the SGPC set-up into the framework employed in this subsection.
Let $\bar{G} := \nabla \hPhi$, $Y:= [0,1]^N$, and $G(\theta, w) := \sum_{i=1}^N w_i\nabla \Phi_i(\theta)$. Moreover, we define $\varepsilon:= 1/\lambda $ and \\ $\xi^{\varepsilon}(t) := e_{\bsi(t)}$, where $e_i$ is the $i$-th unit-vector in $Y$. 
Then, we have $\nabla \Phi_{\bsi(t)} = G(\cdot, \xi^{\varepsilon}(t))$. Assumption~\ref{Assum_kushner}(i) is now immediately implied by Assumption~\ref{Ass_continuity}\revised{; note that any continuous function on $X =\mathbb{R}^K$ is bounded on a bounded subset of $X$}.
The tightness in Assumption~\ref{Assum_kushner}(iii) follows from \revised{$(\xi^{\varepsilon}(t))_{t \geq 0}$ being a  c\`adl\`ag process taking values in the finite set $\{e_i : i \in I \}$; see Theorem 16.8 from \cite{Billingsley1999}}.
To show Assumption~\ref{Assum_kushner}(ii), we employ the explicit representation of the transition kernel $M_t$ of $(\bsi(t))_{t \geq 0}$ given in \eqref{eq_transitionKernel}. Since $(\xi^{\varepsilon}(t))_{t \geq 0}$ is a Markov process and homogeneous in time, we assume without loss of generality that $\underline{t} = 0$. Let now $i_0 \in I$. Then, we have for $s \in [0, \overline{t}]$ the following expression for the conditional expectation:
\begin{align*}
    \mathbb{E}[G(x, \xi^{\varepsilon}(s)) &- \bar{G}(x)| \xi^{\varepsilon}(0) = e_{i_0}] \\ &= \sum_{i = 1}^N \left(\frac{1}{N} -\frac{1}{N}\exp(- N s/\varepsilon)\right)G(x, e_{i}) + \exp(- N s/\varepsilon)G(x, e_{i_0}) - \bar{G}(x) \\
    &= \left(G(x,e_{i_0})-\bar{G}(x)\right) \cdot \exp(- N s/\varepsilon).
\end{align*}
Now we integrate the resulting function on $[0, \overline{t}]$:
\begin{align*}
   \int_{0}^{\overline{t}} \mathbb{E}[G(x, \xi^{\varepsilon}(s)) - \bar{G}(x)| \xi^{\varepsilon}(0)= e_{i_0}] \mathrm{d}s  &=  \int_{0}^{\overline{t}} \left(G(x,e_{i_0})-\bar{G}(x)\right) \cdot \exp(-N s/\varepsilon)  \mathrm{d}s \\
   &= \left(G(x,e_{i_0})-\bar{G}(x)\right) \cdot \int_{0}^{\overline{t}}  \exp(-N s/\varepsilon)  \mathrm{d}s \\ &= \left(G(x,e_{i_0})-\bar{G}(x)\right) \cdot \frac{-\varepsilon}{N}\left(\exp(-N\overline{t}/\varepsilon) -1 \right) \rightarrow 0,
\end{align*}
as $\varepsilon \downarrow 0$. Since $i_0$ was arbitrary, we have
$$
\int_{\underline{t}}^{\overline{t}} \mathbb{E}[G(x, \xi^{\varepsilon}(s)) - \bar{G}(x)| \{\xi^{\varepsilon}(s') : s' \leq \underline{t} \}] \mathrm{d}s \rightarrow 0,
$$
almost surely, as $\varepsilon \downarrow 0$, which implies Assumption~\ref{Assum_kushner}(ii).
Finally, we note that $\theta(0) = \zeta(0)$, hence: $x_0^\varepsilon = x_0$, for $\varepsilon > 0$.
\end{proof}

\subsection{Stochastic gradient descent in nature} \label{Subs_Biolo}
PDMPs are popular models  for random or uncertain processes in biological systems; see Chapter 1 of \cite{Rudnicki2017} for an overview. 
In the following, we  briefly discuss a biological system that is modelled by a dynamical system that corresponds to the SGP. 
This model was proposed by Kussell and Leibler \cite{Kussell2005}.
The modelled biological system contains clonal populations that diversify to survive in randomly fluctuating environments.

\paragraph{Diversified bet-hedging.} In the following, we consider clonal populations, such as bacteria or fungi, that live in fluctuating environments, i.e., environments that are subject to temporal change. Examples are the fluctuation of temperature and light during the day-night-cycle or a different supply of nutrients; see  \cite{Arda2019,Canino-Koning2019}. We define the set of environments to be  $I := \{1,\ldots,N\}$. 
Here, populations typically adapt their phenotypes to retain a high fitness in any environment. 
If the fluctuations within $I$ are irregular or even random, the organisms in a  population cannot adapt to the changes in the environment sufficiently fast; see, e.g., \cite{Kussell2005}.
To prevent extinction and retain high fitness in such fluctuating environments, some populations employ so-called \emph{diversified bet-hedging} strategies; see, e.g., \cite{Haccou1995,Olofsson2009,Sasaki1995,Simovich1997}.
That means, rather than relying on homogeneous switching of phenotypes in the population, the population has heterogeneous phenotypes that are developed and switched based on the current environment $i \in I$ or even completely randomly.
\paragraph{A PDMP model.}
Next, we briefly explain the way  Kussell and Leibler \cite{Kussell2005} model the growth of this population and the phenotype distribution among its individuals.
Indeed, there is a set of $N$ phenotypes, which will be identical to $I$. Indeed, the $i$-th phenotype is the one with the highest fitness in environment $i$, for $i \in I$. 
The fluctuation between environments is modelled by a CTMP $(\bsi(t))_{t \geq 0}$ on $I$ with a certain transition matrix.
Let $\theta_0 \in X := \mathbb{R}^N$. 
Here, the $i$-th component $\theta^{(i)}_0$ of $\theta_0$ describes the number of organisms in the population having phenotype $i \in I$. 
Given we are currently in environment $k \in I$, we assume that organisms with phenotype $i$ grow at a rate $f^{(k)}_i \geq 0$ and that organisms switch from phenotype $i$ to $j$ at rate $H_{j,i}^{(k)}$.
Knowing this, we define
\begin{align*}
G_k := \mathrm{diag}(f^{(k)}) + H^{(k)}  = \begin{pmatrix}f^{(k)}_1+H_{1,1}^{(k)} & H_{1,2}^{(k)} &  \cdots  & H_{1,N}^{(k)} \\
H_{2,1}^{(k)} & f^{(k)}_2+H_{2,2}^{(k)} & \ddots  &\vdots  \\
\vdots & \ddots & \ddots & H_{N-1,N}^{(k)} \\
H_{N,1}^{(k)} & \cdots & H_{N,N-1}^{(k)} & f_N^{(k)} + H_{N,N}^{(k)} 
\end{pmatrix},
\end{align*}
where  $H_{i,i}^{(k)} = - \sum_{j \in I \backslash \{i\}}^N H_{j,i}$, for $i \in I$.
Given an initial vector $\theta_0 \in (0, \infty)^N$ of phenotypes, we can now model the amount of organisms with a particular phenotype via the dynamical system
\begin{equation}\label{Eq_Biol_PDMP}
\frac{\mathrm{d}\theta(t)}{\mathrm{d}t} = G_{\bsi(t)}\theta(t) , \qquad \qquad \theta(0) = \theta_0.
\end{equation}
The dynamical system \eqref{Eq_Biol_PDMP} is a Markov switching process closely related to SGP.
Indeed, we have a homogeneous ODE the right-hand side of which is switched according to a CTMP.

The different environments in the population model represent the different subsamples of the data set that are trained with SGP. 
While the population aims to reach a high fitness in the current environment, SGP aims to optimise an underlying model with respect to the partition of the data set that is currently subsampled.
Overall, SGP aims at solving a certain optimisation problem. In general there is not ad hoc an equivalent optimisation problem in the population dynamic:
Positive growth rates $(f^{(k)})_{k \in I}$ should lead to $$\sum_{j \in I} \theta^{(j)}(t) \rightarrow \infty,$$ as $t \rightarrow \infty$. Moreover, the flows in \eqref{Eq_Biol_PDMP} are likely no gradient flows with underlying scalar potential.
However, diversified bet-hedging strategies also overall aim at long-term high fitness; see \cite{Olofsson2009}.
Hence, both, SGP and diversified bet-hedging aim to enhance a system by enhancing this system in randomly switching situations.
Therefore, we believe that bet-hedging gives a good background for interpreting SGP. 


\section{Long-time behaviour} \label{Sec_longtime}
PDMPs have been subject of extensive studies throughout the last decades, ever since they were introduced by Davis \cite{Davis84}. Many of the results derived in the past also apply to SGP. 
Hence, the PDMP view of SGD gives us access to a large set of analytical tools. Those allow us to study  mixing properties or the long-time behaviour of the algorithm, such as convergence to stationary distributions and ergodicity.

In the following, we will use tools provided by \cite{Bakhtin2012,benaim2012:quant,benaim2015:qual,Cloez2015,Kushner1984} to study the long-time behaviour of SGP. 
Indeed, we will give assumptions under which the processes generated by SGPC and SGPD have a unique stationary measure and are \revised{ ergodic or } exponentially ergodic. For SGPD, we discuss especially the convergence to the minimum of $\hPhi$.
After proving our assertions, we discuss the required assumptions regarding linear least squares estimation problem\revised{s}.
\subsection{Preliminaries}
We collect some notation and basic facts that will be required in the following.
First, we define a distance measure on $X$ for some $q \in (0,1]$:
\begin{equation}
    {d}'( \theta, \theta') := \min \{1,  \|\theta - \theta' \|^q\} \qquad (\theta, \theta' \in X). \label{Eq_metric}
\end{equation}
Note that $d'$ is a metric on $X$ and $(X, d')$ forms a Polish space, i.e. it is separable and complete.
Let  $\pi, \pi'$ be two probability measures on $(X, \mathcal{B}X)$. We define the \emph{Wasserstein(-1) distance} between those measures by $${d}_{\mathrm{W}}(\pi, \pi') := \inf_{H \in \mathrm{Coup}(\pi, \pi')} \int_{X \times X}d'(\chi, \chi') \mathrm{d}H(\chi, \chi'),$$
where $\mathrm{Coup}(\pi, \pi')$ is the set of \emph{couplings} of $\pi, \pi'$. This is the set of probability measures $H$ on $(X \times X, \mathcal{B}X \otimes \mathcal{B}X)$, with $H(\cdot \times X) = \pi$ and $H(X \times \cdot) = \pi'$. Note that due to the boundedness of $d'$, the distance ${d}_{\mathrm{W}}$ is well-defined for any two $\pi, \pi'$ probability measures on $(X, \mathcal{B}X)$. Indeed, the boundedness of $d'$ also implies that convergence in ${d}_{\mathrm{W}}$ is equivalent to weak convergence on $(X, \mathcal{B}X)$. Finally, note that $d'$ being a metric implies that ${d}_{\mathrm{W}}$ is a metric as well. For details see Chapter 6 in the book by Villani \cite{Villani2009}.
Moreover, we define the \emph{Dirac measure} concentrated in $\theta_0 \in X$ by $\delta(\cdot - \theta_0) := \mathbf{1}[\theta_0 \in \cdot].$

Next, we define the \emph{flow} $\varphi_i: X \times [0, \infty) \rightarrow X$ associated to the $i$-th potential $\Phi_i$, for $i \in I$. In particular, $\varphi_i$ satisfies
\begin{align*}
    \frac{\mathrm{d}\varphi_i(\theta_0, t)}{\mathrm{d}t} = - \nabla \Phi_i\left(\varphi_i(\theta_0, t)\right), \qquad \qquad \varphi_i(\theta_0, 0) = \theta_0,
\end{align*}
for any $i \in I$ and $\theta_0 \in X$. Similarly, we  define the Markov kernels associated with the processes $(\theta(t))_{t > 0}$ and $(\xi(t))_{t > 0}$:
\begin{align*}
    \mathrm{C}_t(B|\theta_0, i_0) &= \mathbb{P}(\theta(t) \in B | \theta(0) = \theta_0, \bsi(0) = i_0)   \qquad \qquad (B \in \mathcal{B}X, i_0 \in I, \theta_0 \in X), \\
    \mathrm{D}_{t|t_0}(B|\xi_0, j_0) &= \mathbb{P}(\xi(t) \in B | \xi(t_0) = \xi_0, \bsj(t_0) = j_0)  \qquad \qquad(B \in \mathcal{B}X, j_0 \in I, \xi_0 \in X),
\end{align*}
where $t \geq t_0 \geq 0$.
We now note two different assumptions on the convexity of the $\Phi_i$; a weak and a strong version.
\begin{assumption}[Strong convexity]\label{Ass_Contr}
For every $i \in I$, there is a $\kappa_i \in \mathbb{R}$, with
\begin{equation} \label{eq_convexity}
\left\langle \theta_0 - \theta_0', \nabla \Phi_i(\theta_0) - \nabla\Phi_i(\theta_0')\right\rangle \geq\kappa_i\|\theta_0 - \theta_0' \|^2,
\end{equation}
with either
\begin{itemize}
\item[(i)]    $\kappa_1 + \cdots + \kappa_N > 0$ and \revised{for every $\theta_0 \in X$ there is some bounded $S \in \mathcal{B}X$, $S \ni \theta_0$, such that  $$\varphi_i(S,t) \subseteq S \qquad (i \in I, t \geq 0)$$} (weak) or 

 \item[(ii)] $\kappa_1 = \cdots =\kappa_N > 0$ (strong).
\end{itemize}
\end{assumption}
In the strong version, we assume that all of the potentials $\{\Phi_i\}_{i \in I}$ are strongly convex. 
In the weak version, strong convexity of some potentials is sufficient; however, we need to ensure additionally that none of the flows escapes to infinity. \revised{The set $S$, in which we trap the process, is called \emph{positively invariant} for $(\varphi_i)_{i \in I}$. The uniform strong convexity in  Assumption~\ref{Ass_Contr}(ii), indeed, implies the existence of such a set for all $\theta_0 \in X$.  }

Both, Assumption~\ref{Ass_Contr}(i) and (ii) are quite strong. As we have mentioned before,  optimisation problems in machine learning are often non-convex. However, we focus on convex optimisation problems in this study.
Strong convexity implies for instance that the associated flows contract exponentially:
\begin{lemma} \label{Lemma_convex}
Inequality \eqref{eq_convexity} for some $i \in I$ implies that the corresponding flows contract exponentially, i.e.
\begin{equation*} \label{Eq_contra}
\|\varphi_i(\theta_0, t) - \varphi_i(\theta_0', t) \| \leq \exp(-\kappa_i t) \|\theta_0 - \theta_0' \|.
\end{equation*} 
\end{lemma}
\begin{proof}
This is implied by Lemma 4.1 given in \cite{Cloez2015}.
\end{proof}
Given this background, we now study the ergodicity of SGP. We commence with the case of a constant learning rate.

\subsection{Constant learning rate} \label{Subsec_const}
Under Assumption~\ref{Ass_Contr}(i), the SGP  $(\theta(t), \bsi(t))_{t > 0}$ has a unique stationary measure $\pi_{\mathrm{C}}$ on $(Z, \mathcal{B}Z) := (X \times I, \mathcal{B}X \otimes 2^I)$ and it contracts with respect to this measure in the Wasserstein distance $d_{\mathrm{W}}$.
As the Markov process contracts exponentially, we say, the Markov process is \emph{exponentially ergodic}.
We now state this result more particularly:

\begin{theorem} \label{Thm_Hairer}
Let Assumptions~\ref{Ass_continuity} and \ref{Ass_Contr}(i) hold. Then, $(\theta(t), \bsi(t))_{t > 0}$ has a unique stationary measure $\pi_{\mathrm{C}}$ on $(Z, \mathcal{B}Z)$. Moreover, there exist $\kappa', c > 0$ and $q \in (0, 1]$, with
\begin{align*}
d_{\mathrm{W}}(\pi_{\rm C}(\cdot \times I), \mathrm{C}_t(\cdot|\theta_0, i_0))  \leq c \exp(-\kappa' t) \left(1 + \sum_{i \in I}\int_{X} \|\theta_0 - \theta' \|^q\pi_{\mathrm{C}}(\mathrm{d}\theta' \times \{i\})\right)
\end{align*}
for any $i_0 \in I$ and $\theta_0 \in X$.
\end{theorem}
The proof of this theorem follows similar lines as the proof of Theorem~\ref{Thm_Hairer2}. Thus, we prove both in \S\ref{Subsec_Proof_Hairer12}.
Note that in Theorem~\ref{Thm_Hairer}, $q$ influences the metric $d'$ that is defined in \eqref{Eq_metric} and that is part of the Wasserstein distance $d_{\rm W}$.
This result implies that SGPC converges very quickly to its stationary regime. 
For estimates of the constants in Theorem~\ref{Thm_Hairer}, we refer to \cite{benaim2012:quant}.
Determining the stationary measure $\pi_{\mathrm{C}}$ may be rather difficult in practice; see \cite{Costa1990,durmus2018piecewise}. We give numerical illustrations in \S\ref{Sec_NumIll}.

\subsection{Decreasing learning rate} \label{Sec_Longtime_Decr}
Next, we study the longtime behaviour of SGP with decreasing learning rate. Here, we are less interested in the convergence of SGP to some abstract probability measure. Instead,  we study the convergence of SGPD to the minimum $\theta^* \in X$ of the full potential $\hPhi$.
Hence, we aim to analyse the behaviour of 
$$
d_{\mathrm{W}}(\delta(\cdot - \theta^*), \mathrm{D}_{t|0}(\cdot|\xi_0, j_0)),
$$ as $t \rightarrow \infty$. 
Here, we have anticipated that the Dirac measure $\delta(\cdot - \theta^*)$ is the stationary measure of SGPD as $t \rightarrow \infty$. This can be motivated by Theorem~\ref{thm:ODE_limit} where SGPC converges to the full gradient flow, as $\eta \downarrow 0$.

Two aspects of SGPD imply that the analysis of this distance is significantly more involved than that of SGPC.
First, the process is inhomogeneous in time; a case hardly discussed in the literature. We use the following standard idea to solve this issue:
\begin{itemize}
\item[(i)] We  define a homogeneous Markov chain $(\xi'(t))_{t \geq 0}$ on an extended state space $X \times \mathbb{R}$ where the transition rate matrix of $(\bsj(t))_{t \geq 0}$ will not depend on time, but on the current position of $(\xi'(t))_{t \geq 0}$.
\end{itemize}
Second, as $t \rightarrow \infty$ the rate matrix $B(t)$ degenerates; the diagonal entries go to $-\infty$, the off-diagonal entries will go to $\infty$.  This case is not covered by \cite{Cloez2015} or related literature on PDMPs -- to the best of our knowledge. 
However, we were discussing a closely related problem in  Theorem~\ref{thm:ODE_limit}. To apply the perturbed test function theory, we require three fold actions:
\begin{itemize}
\item[(ii)] We define an auxiliary Markov jump process with bounded transition rate matrix. 
\item[(iii)] We show that the PDMP based on this Markov jump process converges to a unique stationary measure at exponential rate.  
\item[(iv)] We show that this stationary measure approximates $\delta(\cdot - \theta^*)$ at any precision. Also, we show that the auxiliary PDMP approximates SGPD.
\end{itemize}
Finally, we will obtain the following result:
\begin{theorem} \label{Thm_LS}
Let Assumptions~\ref{Ass_continuity} and \ref{Ass_Contr}(ii) hold.   Then,  there is a function $\alpha: [0,1) \rightarrow [0, \infty)$ that is continuous at $0$ and satifies $\alpha(0)= 0$. Moreover, for any $\varepsilon \in (0,1)$, we have a  probability measure $\pi_\varepsilon$ and constants $\kappa', c > 0$, $q \in (0, 1]$  such that 
\begin{align*}
d_{\mathrm{W}}(\delta(\cdot - \theta^*), \mathrm{D}_{t|0}(\cdot|\xi_0, j_0)) \leq c \exp(-\kappa' t) \left(1 + \sum_{j \in I}\int_{X} \|\xi_0 - \xi'' \|^q\pi_\varepsilon(\mathrm{d}\xi'' \times \{j\})\right) + \alpha(\varepsilon),
\end{align*}
for any $j_0 \in I$ and $\xi_0 \in X$.
\end{theorem}
Hence, as $t \rightarrow \infty$, the state $\xi(t)$ of the SGPD converges weakly to the Dirac measure concentrated in the minimum $\theta^*$ of the full target function $\hPhi$.
To prove this theorem, we now walk  through steps (i)-(iv). 
Using several auxiliary results, we are then able to give a proof of Theorem~\ref{Thm_LS}.
\paragraph{(i) A homogeneous formulation.}
We now formulate the SGPD in a time-homogeneous fashion. Indeed, we define $(\xi'(t))_{t \geq 0} := (\xi(t), \tau(t))_{t \geq 0}$, with
\begin{align*}
\frac{\mathrm{d}\xi'(t)}{\mathrm{d}t} &= 
\begin{pmatrix}
\frac{\mathrm{d}\xi(t)}{\mathrm{d}t} \\
\frac{\mathrm{d}\tau(t)}{\mathrm{d}t} 
\end{pmatrix} 
= \begin{pmatrix}
- \revised{\nabla}\Phi_{\bsj(t)}(\xi(t)) \\
-\tau(t) \end{pmatrix} =: \revised{\vec{\Psi}}_{\bsj(t)}(\xi'(t)), \\ \xi'(0) &=  \begin{pmatrix}
\xi(0) \\
\tau(0) 
\end{pmatrix} = \begin{pmatrix}
\xi_0 \\
1
\end{pmatrix} =: \xi'_0
\end{align*}
and $(\bsj(t))_{t \geq 0}$ has transition rate matrix $$B'(\cdot) := B(-\log(\tau)).$$ One can see easily that this definition of SGPD is equivalent to our original Definition~\ref{Def_SGD_MP}(ii).
Note furthermore that the dynamic is defined such that if $\{ \nabla\Phi_i \}_{i \in I }$ satisfies Assumption \ref{Ass_Contr}(i) (resp. (ii)) $\{ \revised{\vec{\Psi}}_i \}_{i \in I }$ does as well.
\paragraph{(ii) An auxiliary PDMP.} Let $\varepsilon \in (0,1)$. We define the PDMP $(\bsj_\varepsilon(t), \xi_\varepsilon(t))_{t \geq 0}$ by
$$
\begin{pmatrix}
\frac{\mathrm{d}\xi_\varepsilon(t)}{\mathrm{d}t} \\
\frac{\mathrm{d}\tau_\varepsilon(t)}{\mathrm{d}t} 
\end{pmatrix} 
= \begin{pmatrix}
- \revised{\nabla}\Phi_{\bsj_\varepsilon(t)}(\xi(t)) \\
\varepsilon-\tau_{\varepsilon}(t) \end{pmatrix}, \qquad \begin{pmatrix}
\xi_\varepsilon(0) \\
\tau_\varepsilon(0) 
\end{pmatrix} = \begin{pmatrix}
\xi_0 \\
1
\end{pmatrix},
$$
where the Markov jump process $(\bsj_\varepsilon(t))_{t \geq 0}$ has transition rate matrix $B_\varepsilon(\cdot) := B(-\log(\tau_\varepsilon)).$ Note that -- as opposed to $B(\cdot)$ -- this transition rate matrix converges to 
$B(-\log(\varepsilon))$, as $t \rightarrow \infty$. Moreover, we define the Markov transition kernel of $(\xi_\varepsilon(t))_{t \geq 0}$ by $\mathrm{D}^{\varepsilon}_{t|t_0}$.

\paragraph{(iii) Ergodicity of the auxiliary process.}
The following theorem shows that the auxiliary process $(\xi_\varepsilon(t), \bsj_\varepsilon(t))_{t \geq 0}$ converges at exponential rate to its unique stationary measure.
\begin{theorem} \label{Thm_Hairer2}
Let Assumptions~\ref{Ass_continuity} and \ref{Ass_Contr}(ii) hold and let $\varepsilon > 0$. Then, $(\xi_\varepsilon(t), \bsj_\varepsilon(t))_{t > 0}$ has a unique stationary measure $\pi_{\varepsilon}$ on $(Z, \mathcal{B}Z)$. Moreover, there exist $\kappa', c > 0$ and $q \in (0, 1]$, with
\begin{align*}
d_{\mathrm{W}}(\pi_{\varepsilon}( \cdot \times I), \mathrm{D}^\varepsilon_{t|0}(\cdot|\xi_0, j_0))  \leq  c \exp(-\kappa' t) \left(1 + \sum_{j \in I}\int_{X} \|\xi_0 - \xi'' \|^q\pi_\varepsilon(\mathrm{d}\xi'' \times \{j\})\right) 
\end{align*}
for any $j_0 \in I$ and $\xi_0 \in X$.
\end{theorem}
As mentioned before, we give the proof of Theorem~\ref{Thm_Hairer2} in \S\ref{Subsec_Proof_Hairer12}.
Note that we now require Assumption~\ref{Ass_Contr}(ii), i.e., the strong version.

\paragraph{(iv) Weak convergence of the auxiliary process.} The last preliminary step consists in showing that the auxiliary process $(\xi_\varepsilon(t))_{t \geq 0}$ approximates the SGPD $(\xi(t))_{t \geq 0}$. Moreover, the same needs to hold for the respective stationary measures.
\begin{proposition} \label{Propo_conv_auxiliary_pro}
Let Assumptions~\ref{Ass_continuity} and \ref{Ass_Contr}(ii) hold. Then,
\begin{itemize}
\item[(i)] there is a  function $\alpha': [0,1) \rightarrow [0, \infty)$, that is continuous at $0$ and satisfies $\alpha'(0) = 0$, such that 
$$d_{\mathrm{W}}(\mathrm{D}^{\varepsilon}_{t|0}(\cdot|\xi_0,j_0), \mathrm{D}_{t|0}(\cdot|\xi_0,j_0)) \leq \alpha'(\varepsilon),$$
for any $j_0 \in I, \xi_0 \in X, t \geq t_0 \geq 0$,
\item[(ii)] there is a  function $\alpha'': [0,1) \rightarrow [0, \infty)$,  that is continuous at $0$ and satisfies $\alpha''(0) = 0$, such that  
$$d_{\mathrm{W}}(\delta(\cdot - \theta^*), \pi_\varepsilon(\cdot \times I)) \leq \alpha''(\varepsilon)
$$
\end{itemize}
\end{proposition}
The proof of Proposition~\ref{Propo_conv_auxiliary_pro} is more involved. We present our proof along with several auxiliary results in \S\ref{Appendix_Prop_approx}.

Given the results in (i)-(iv), we can proceed to proving the main result.
\begin{proof}[Proof of Theorem \ref{Thm_LS}]
Note that by the triangle inequality, we have
\begin{align*}
d_{\mathrm{W}} (\delta(\cdot - \theta^*), \mathrm{D}_{t|0}(\cdot|\xi_0, j_0))  \leq &d_{\mathrm{W}}(\delta(\cdot - \theta^*), \pi_\varepsilon(\cdot \times I)) \\&+  d_{\mathrm{W}}(\pi_{\varepsilon}( \cdot \times I), \mathrm{D}^\varepsilon_{t|0}(\cdot|\xi_0, j_0)) \\&+ d_{\mathrm{W}}(\mathrm{D}^{\varepsilon}_{t|0}(\cdot|\xi_0,j_0), \mathrm{D}_{t|0}(\cdot|\xi_0,j_0)).
\end{align*}
Now, we employ Theorem~\ref{Thm_Hairer2} and \revised{obtain \begin{align*}
d_{\mathrm{W}}&(\pi_{\varepsilon}( \cdot \times I), \mathrm{D}^\varepsilon_{t|0}(\cdot|\xi_0, j_0))  \leq  c \exp(-\kappa' t) \left(1 + \sum_{j \in I}\int_{X} \|\xi_0 - \xi'' \|^q\pi_\varepsilon(\mathrm{d}\xi'' \times \{j\})\right), 
\end{align*} for some $\kappa',c > 0$ and $q \in (0,1]$. Moreover, with} Proposition~\ref{Propo_conv_auxiliary_pro}\revised{, we can bound $$d_{\mathrm{W}}(\delta(\cdot - \theta^*), \pi_\varepsilon(\cdot \times I)) \leq \alpha''(\varepsilon)$$ and $$d_{\mathrm{W}}(\mathrm{D}^{\varepsilon}_{t|0}(\cdot|\xi_0,j_0), \mathrm{D}_{t|0}(\cdot|\xi_0,j_0)) \leq \alpha'(\varepsilon).$$} We finally obtain our assertion setting $\alpha := \alpha' + \alpha''$. 
\end{proof}

\subsection{Proofs of Theorem~\ref{Thm_Hairer} and Theorem~\ref{Thm_Hairer2}} \label{Subsec_Proof_Hairer12}

The proof of Theorem~\ref{Thm_Hairer} proceeds by showing the assumptions of Theorem 1.4 in \cite{Cloez2015}, which implies exponential ergodicity of the PDMP. Under the same assumptions, Corollary 1.11 of \cite{benaim2012:quant} implies uniqueness of the stationary measure.
We denote the necessary assumptions below, then we proceed with the proof.
\begin{assumption} \label{Ass_Hairer} We consider the following three assumptions:
\begin{itemize}
\item[(i)] the process $(\bsi(t))_{t \geq 0}$ is non-explosive, irreducible and positive recurrent,
\item[(ii)] the Markov kernels representing the different gradient flows ${\rm C}^{(i)}_t(\cdot|\theta_0) := \delta(\cdot - \varphi_i(\theta_0,t))$ are on average exponentially contracting in $d_{\rm W}$, i.e.
for any two probability measures $\pi, \pi'$ on $(X, \mathcal{B}X)$ satisfy
$$
d_{\rm W}(\pi {\rm C}^{(i)}_t, \pi' {\rm C}^{(i)}_t) \leq \exp(-\kappa_it) d_{\rm W}(\pi, \pi') \quad (i \in I)
$$
for any $t > 0$ and $\kappa_1 + \cdots +\kappa_N > 0$, and
\item[(iii)] the Markov kernel ${\rm C}_t$  has a finite first absolute moment, i.e.
$$\frac{1}{N}\sum_{i_0=1}^N\int \|\theta\| {\rm C}_t(\mathrm{d}\theta|\theta_0,i_0) < \infty,$$
for $t \geq 0$ and $\theta_0 \in X$.
\end{itemize}
\end{assumption}

\begin{proof}[Proof of Theorem~\ref{Thm_Hairer}]
Assumption~\ref{Ass_Hairer}(i) is satisfied by standard properties of \revised{homogeneous} continuous-time Markov processes on finite sets.
Assumption~\ref{Ass_Hairer}(ii) is implied by Assumption~\ref{Ass_Contr}(i); see also the proof of Lemma 2.2 in \cite{Cloez2015}: 
Let $G$ be \revised{a coupling} in $\mathrm{Coup}(\pi, \pi')$ and choose \revised{a coupling} $H \in \mathrm{Coup}(\pi {\rm C}^{(i)}_t, \pi' {\rm C}^{(i)}_t)$, such that
\begin{align*}
\int_{X \times X}d'(\chi, \chi') \mathrm{d}H(\chi, \chi') = \ \int_{X \times X}d'(\varphi_{i}(\chi,t), \varphi_{i}(\chi',t)) \mathrm{d}G(\chi, \chi').  
\end{align*}
By Assumption~\ref{Ass_Contr}(i) and Lemma~\ref{Lemma_convex}, we have 
\begin{align*}
\int_{X \times X}d'(\varphi_{i}(\chi,t), \varphi_{i}(\chi',t)) \mathrm{d}G(\chi, \chi') \leq  \exp(-\kappa_it) \int_{X \times X}d'(\chi, \chi') \mathrm{d}G(\chi, \chi')
\end{align*}
Thus, we have indeed the required contractivity in the Wasserstein distance:
\begin{align*}
d_{\rm W}(\pi {\rm C}^{(i)}_t, \pi' {\rm C}^{(i)}_t)  &\leq \int_{X \times X}d'(\chi, \chi') \mathrm{d}H(\chi, \chi')   \leq \exp(-\kappa_it) \int_{X \times X}d'(\chi, \chi') \mathrm{d}G(\chi, \chi')
\end{align*}
As $d_{\rm W}(\pi {\rm C}^{(i)}_t, \pi' {\rm C}^{(i)}_t)$ does not depend on $H$ and $G$, we finally obtain
$$
d_{\rm W}(\pi {\rm C}^{(i)}_t, \pi' {\rm C}^{(i)}_t)  \leq \exp(-\kappa_it) d_{\rm W}(\pi, \pi').
$$
Concerning Assumption~\ref{Ass_Hairer}(iii), we employ the boundedness of the flows in Assumption~\ref{Ass_Contr}(i). 
\end{proof}

Now we move on to the proof of Theorem~\ref{Thm_Hairer2}. It is conceptually similar to the proof of Theorem~\ref{Thm_Hairer}: It relies on proving the necessary assumptions of Theorem 3.3 in \cite{Cloez2015}. The uniqueness of the stationary measure follows under the same assumptions from Corollary 1.16 in \cite{benaim2012:quant}. 
We state these assumptions below.
\begin{assumption}\label{Ass_Hairer2} We consider the following four assumptions:
\begin{itemize}
\item[(i)]  there is a transition rate matrix $\underline{B}$ leading to a positive recurrent, irreducible Markov chain and 
$$
\underline{B}_{i,j} = \inf_{\tau \in (\varepsilon, 1]} B(\revised{-\log(\tau)})_{i,j} \qquad (i, j \in I, i \neq j),
$$
\item[(ii)]  the Markov kernels representing the different gradient flows ${\rm C}^{(i)}_t(\cdot|\theta_0) := \delta(\cdot - \varphi_i(\theta_0,t))$ are exponentially contracting in $d_{\rm W}$, i.e.
for any two probability measures $\pi, \pi'$ on $(X, \mathcal{B}X)$ satisfy
$$
d_{\rm W}(\pi {\rm C}^{(i)}_t, \pi' {\rm C}^{(i)}_t) \leq \exp(-\kappa_it) d_{\rm W}(\pi, \pi') \quad  (i \in I)
$$
for any $t > 0$ and $\kappa_1 = \cdots =\kappa_N > 0$,
\item[(iii)] the Markov kernel ${\rm D}^{\varepsilon}_{t|0}$ has a finite first absolute moment, i.e.
$$\frac{1}{N}\sum_{\revised{j}_0=1}^N\int \|\xi\|{\rm D}^{\varepsilon}_{t|0}(\mathrm{d}\xi|\xi_0,j_0) < \infty,$$
for $t\geq 0$ and $\xi_0 \in X$,
and
\item[(iv)] the transition rate matrix $B_{\varepsilon}$ is bounded in the sense that 
\begin{equation}
\sup_{\tau \in (\varepsilon,1]} \sup_{i \in I}\sum_{j \in I} B_{\varepsilon}(\tau)_{i,j} \mathbf{1}[i \neq j]   < \infty
\end{equation}
and  $\sup_{i \in I}\sum_{j \in I} B_{\varepsilon}(\tau)_{i,j} \mathbf{1}[i \neq j]$ is Lipschitz continuous with respect to $\tau \in (\varepsilon, 1]$.
\end{itemize}
\end{assumption}
Note that Assumption~\ref{Ass_Hairer2}(i)-(iii) closely correspond to Assumption~\ref{Ass_Hairer}.
\begin{proof}[Proof of Theorem~\ref{Thm_Hairer2}]
The matrix $\underline{B}$ in Assumption~\ref{Ass_Hairer2}(i) is given by $$
\underline{B}_{i,j}  = \mu(0), \quad \underline{B}_{i,i}  = -(N-1)\cdot \mu(0) \quad (i, j \in I, i \neq j)
$$
which indeed induces positive recurrent, irreducible conti\-nuous-time Markov chain.
Assumption~\ref{Ass_Hairer2}(ii) can be proven analogously to Assumption~\ref{Ass_Hairer}(ii) in the proof of Theorem~\ref{Thm_Hairer}. Here, Assumption~\ref{Ass_Contr}(i) is replaced by Assumption~\ref{Ass_Contr}(ii).
Next, we prove Assumption (iv). First, note that we can write 
\begin{align*}
\sup_{i \in I}\sum_{j \in I} B_{\varepsilon}(t)_{i,j} \mathbf{1}[i \neq j] = (N-1) \mu(\tau_{\varepsilon}(t)) = (N-1) \mu(-\log(\tau)).
\end{align*}
Boundedness and Lipschitz continuity of this function, follows from the boundedness of $\tau \in (\varepsilon, 1]$ and the \revised{continuous differentiability} of $\mu$.
Assumption~\ref{Ass_Hairer2}(iii) follows from Lemma 1.14 in \cite{benaim2012:quant} and Assumption~\ref{Ass_Hairer2}(iv).
\end{proof}
\subsection{Proof of Proposition~\ref{Propo_conv_auxiliary_pro}} \label{Appendix_Prop_approx}
In this subsection, we prove Proposition~\ref{Propo_conv_auxiliary_pro}.
First, we show weak convergence of $(\xi_{\varepsilon}(t))_{t \geq 0} \Rightarrow (\xi(t))_{t \geq 0}$ in the sense of \eqref{eq_weak_conv}.
Given this result, we will be able to construct the function $\alpha'$ and thus prove  Proposition~\ref{Propo_conv_auxiliary_pro}(i). Part (ii) of the proposition will rely on showing that $(\xi_{\varepsilon}(t))_{t \geq 0}$ approximates the underlying gradient flow, as discussed in Theorem~\ref{thm:ODE_limit}.
\begin{lemma} \label{lemma_weak_conv}
Let Assumption\revised{s}~\ref{Ass_continuity} \revised{and \ref{Ass_Contr}(ii)} hold. Then, $$(\xi_{\varepsilon}(t), \tau_{\varepsilon}(t), \bsj_{\varepsilon}(t))_{t \geq 0} \Rightarrow (\xi(t), \tau(t), \bsj(t))_{t \geq 0},$$ as $\varepsilon \downarrow 0$.
\end{lemma}
\begin{proof}
Let $Z' := X \times \mathbb{R} \times \mathbb{R}$, let $\mathcal{A}$ be the (infinitesimal) generator of $(\xi(t), \tau(t), \bsj(t))_{t \geq 0}$, and let analogously $\mathcal{A}_{\varepsilon}$ be the generator of $(\xi_{\varepsilon}(t), \tau_{\varepsilon}(t), \bsj_{\varepsilon}(t))_{t \geq 0} $, for any $\varepsilon > 0$.
We will now employ Theorem 3.2 of \cite{Kushner1984} which implies our assertion, if
\begin{itemize}
\item[(i)] the family $(\xi_{\varepsilon}(t))_{t \geq 0, \varepsilon > 0}$ is tight with respect to $\varepsilon$,
\item[(ii)] for any $T \in (0, \infty)$ and any test function $f \in C'$ there is a `perturbed' test function $f^{\varepsilon}: [0, \infty) \rightarrow \mathbb{R}$, such that
\begin{align}
\sup_{\substack{t \geq 0 \\ \varepsilon \in (0,1]}}\mathbb{E}\left[\left\lvert f^\varepsilon(t)-f(\xi_{\varepsilon}(t), \tau_{\varepsilon}(t), \bsj_{\varepsilon}(t))\right\rvert \right] &< \infty, \label{eq_proofprop1st} \\
\lim_{\varepsilon \downarrow 0}\mathbb{E}\left[\left\lvert f^\varepsilon(t)-f(\xi_{\varepsilon}(t), \tau_{\varepsilon}(t), \bsj_{\varepsilon}(t)) \right\rvert \right] &=0  \qquad (t \geq 0), \label{eq_proofprop2nd} \\
\sup_{\substack{t \in (0, T] \\ \varepsilon \in (0,1]}}\mathbb{E}\left[\left\lvert \mathcal{A}_{\varepsilon}f^\varepsilon(t)-\mathcal{A}f(\xi_{\varepsilon}(t), \tau_{\varepsilon}(t), \bsj_{\varepsilon}(t))\right\rvert \right] &< \infty,  \label{eq_proofprop3rd} \\
\lim_{\varepsilon \downarrow 0}\mathbb{E}\left[\left\lvert \mathcal{A}_{\varepsilon}f^\varepsilon(t)-\mathcal{A}f(\xi_{\varepsilon}(t), \tau_{\varepsilon}(t), \bsj_{\varepsilon}(t)) \right\rvert \right] &=0  \qquad (0 \leq t \leq T). \label{eq_proofprop4th}
\end{align}
Here, $C'$ is uniformly dense in the space  $C^0_c(Z')$ of continuous functions with compact support.
\end{itemize}
First, note that  the generators are given by 
\begin{align*}
\mathcal{A}f(\xi,\tau,i) &:= \left\langle \begin{pmatrix}
- \revised{\nabla}\Phi_{i}(\xi) \\
-\tau \end{pmatrix}, \nabla_{\xi,\tau}f(\xi,\tau,i)\right\rangle  + \mu(-\log(\tau))\sum_{j \in I} \left( f(\xi, \tau, j) - f(\xi, \tau,i) \right), \\
\mathcal{A}_{\varepsilon}f(\xi,\tau,i) &:= \left\langle \begin{pmatrix}
- \revised{\nabla}\Phi_{i}(\xi) \\
\varepsilon-\tau \end{pmatrix}, \nabla_{\xi,\tau}f(\xi,\tau,i)\right\rangle  + \mu(-\log(\tau))\sum_{j \in I} \left( f(\xi, \tau, j) - f(\xi, \tau,i) \right),
\end{align*}
for any $f: Z' \rightarrow \mathbb{R}$ that is twice continuously differentiable and vanishes at infinity;
see, e.g., \cite{Davis84} for details. 
Here, we understand the processes $(\xi(t), \tau(t), \bsj(t))_{t \geq 0}$ and $(\xi_\varepsilon(t), \tau_{\varepsilon}(t), \bsj_{\varepsilon}(t))_{t \geq 0}$ as Markov jump diffusions.
Tightness in (i) follows from the boundedness of the \revised{gradient} in Assumption~\ref{Ass_continuity}\revised{: According to Theorem 2.4 in \cite{Kushner1984} (or, e.g., Theorem 7.3 in \cite{Billingsley1999}), we need to show that (i1), (i2) are satisfied by $(\xi_{\varepsilon}(t))_{t \geq 0}$:
\begin{itemize}
\item[(i1)] For all $\eta_* > 0$, there is an $N_* \in (0, \infty)$, with
$$
\mathbb{P}(\|\xi_{\varepsilon}(0)\| \geq N_*) \leq \eta_* \qquad (\varepsilon>0).
$$
\item[(i2)] For all $\eta_* >0$, $\varepsilon_*>0, \overline{t}>0$ there is $\delta_* > 0$ and an $n_0 \in (0, \infty)$, such that
$$
\mathbb{P}\left(\sup_{|s-t|< \delta_*, 0 \leq s \leq t \leq \overline{t}}\|\xi_{\varepsilon}(t)-\xi_{\varepsilon}(s)\| \geq \varepsilon_*\right)\leq \eta_*,$$ for $\varepsilon \in (0, n_0)$. 
\end{itemize} 
(i1) is satisfied as the initial value $\xi_{\varepsilon}(0)$ is $\mathbb{P}$-a.s. constant throughout $\varepsilon
>0$. To prove (i2), note that $(\xi_{\varepsilon}(t))_{t \geq 0}$ has $\mathbb{P}$-a.s. continuous paths that are almost everywhere differentiable. Let  $B \subseteq X$ be a closed ball with $\mathbb{P}(\xi_{\varepsilon}(t) \in B) = 1$ $(t \geq 0)$; see Lemma 1.14 in \cite{benaim2012:quant}. The derivative of $(\xi_{\varepsilon}(t))_{t \geq 0}$ is  bounded by some finite $$L \geq \sup_{i \in I, \theta_0 \in B}\|\nabla \Phi_i(\theta) \|,$$ as the $(\nabla \Phi_i)_{i \in I}$  are continuous. Importantly, $L$ does not depend on $\varepsilon$. 
Hence, we have
$$
\|\xi_{\varepsilon}(t)-\xi_{\varepsilon}(s)\| \leq L |t-s| 
$$
$\mathbb{P}$-a.s. for $0 \leq s \leq t$. This implies 
$$
\sup_{|s-t| < \delta_*, 0\leq s\leq t}\|\xi_{\varepsilon}(t)-\xi_{\varepsilon}(s)\| \leq L \delta_* 
$$
$\mathbb{P}$-a.s. for any $\delta_{*}>0$.
Thus, we get for any $\varepsilon_*>0, \overline{t}>0$: $\delta_* := \varepsilon_*/L$ and
\begin{align*}
\mathbb{P}&\left(\sup_{|s-t| <\delta_*, 0 \leq s \leq t \leq \overline{t}}\|\xi_{\varepsilon}(t)-\xi_{\varepsilon}(s)\| \leq  \varepsilon_*\right) =1.
\end{align*}
This implies 
\begin{align*}
\mathbb{P}&\left(\sup_{|s-t| < \delta_*, 0 \leq s \leq t \leq \overline{t}}\|\xi_{\varepsilon}(t)-\xi_{\varepsilon}(s)\| > \varepsilon_*\right) =0,
\end{align*}
which means
\begin{align*}
0 = \mathbb{P}\left(\sup_{|s-t| < 2\delta_*, 0 \leq s \leq t \leq \overline{t}}\|\xi_{\varepsilon}(t)-\xi_{\varepsilon}(s)\| \geq 2\varepsilon_*\right)  \geq \mathbb{P}\left(\sup_{|s-t|\leq 2\delta_*, 0 \leq s \leq t \leq \overline{t}}\|\xi_{\varepsilon}(t)-\xi_{\varepsilon}(s)\| > \varepsilon_*\right)
\end{align*}
giving us (i2).
}

To prove (ii), we choose the test space $C' := C^2_{\revised{c}}(Z')$, which
is the space of twice continuously differentiable functions that \revised{have compact support} and that have bounded $C^2$-sup-norm. Note that the Stone-Weierstrass Theorem for locally compact $Z'$ implies that $C^2_{\revised{c}}(Z')$ is uniformly dense in $C^0_0(Z')$\revised{; see, e.g., Corollary 4.3.5 in \cite{Pedersen1989}}. Thus, $C^2_{\revised{c}}(Z')$ is also uniformly dense in $C^0_c \subseteq C^0_0$.

Now, for any test function $f \in C'$ we choose the perturbed test function $f^\varepsilon(t) := f(\xi_{\varepsilon}(t))$, $t \geq 0, \varepsilon \in (0,1]$. Then, 
we have $f^\varepsilon-f(\xi_{\varepsilon}) \equiv 0$, for any $\varepsilon \in (0,1]$. Hence, \eqref{eq_proofprop1st} and \eqref{eq_proofprop2nd} are satisfied. Now towards \eqref{eq_proofprop3rd} and \eqref{eq_proofprop4th}. For $\varepsilon > 0$ and $t \in [0, T]$, we compute 
\begin{align*}
\mathcal{A}_{\varepsilon}f^\varepsilon(t)-&\mathcal{A}f(\xi_{\varepsilon}(t), \tau_{\varepsilon}(t), \bsj_{\varepsilon}(t)) = \varepsilon \cdot \frac{\partial}{\partial \tau}f(\xi_{\varepsilon}(t), \tau_{\varepsilon}(t), \bsj_{\varepsilon}(t)).
\end{align*}
By assumption the partial derivatives of $f$ are bounded. Hence, we obtain
$$\mathbb{E}\left[\left\lvert \mathcal{A}_{\varepsilon}f^\varepsilon(t)-\mathcal{A}f(\xi_{\varepsilon}(t), \tau_{\varepsilon}(t), \revised{\bsj}_{\varepsilon}(t))\right\rvert \right]  \leq \varepsilon \sup_{z' \in Z'}\left\lvert\frac{\partial f(z')}{\partial \tau}\right\rvert,$$
where the supremum on the right-hand side is finite, as $f \in C'$.
This proves \eqref{eq_proofprop3rd}, \eqref{eq_proofprop4th} and concludes the proof.
\end{proof}
We can now employ Lemma~\ref{lemma_weak_conv} to find an appropriate bound for the Wasserstein distances in the first part of Proposition~\ref{Propo_conv_auxiliary_pro}. 
\begin{proof}[{Proof of Proposition~\ref{Propo_conv_auxiliary_pro} (i)}]
From Lemma~\ref{lemma_weak_conv}, we know that $(\xi_{\varepsilon}(t), \tau_{\varepsilon}(t), \bsj_{\varepsilon}(t))_{t \geq 0} \Rightarrow (\xi(t), \tau(t), \bsj(t))_{t \geq 0}$, as $\varepsilon \downarrow 0$. Note that this is equivalent to $(\xi_{\varepsilon}(t), \tau_{\varepsilon}(t), \bsj_{\varepsilon}(t))_{t \geq 0} - (\xi(t), \tau(t), \bsj(t))_{t \geq 0}  \Rightarrow 0$.
We now construct the function $\alpha'(\cdot)$. Let
$$
F(\xi,\tau, \bsj) :=  \left(\sup_{t \geq 0} \min\{1, \| \xi(t) \|\}\right)^q,
$$
where   $ (\xi, \tau, \bsj) \in C^0([0, \infty); Z')$.
$F$ is bounded and continuous on $(C^0([0, \infty); Z), \|\cdot \|_\infty)$, since
$$F(\xi, \tau, \bsj)= \begin{cases}
1, &\text{ if } \|\xi\|_{\infty} > 1, \\
\| \xi\|_{\infty}^q, &\text{ if } \|\xi \|_{\infty} \leq 1
\end{cases}$$
is continuous for any $(\xi, \tau, \bsj) \in C^0([0, \infty); Z')$. 
 The weak convergence of $$(\xi_{\varepsilon}(t), \tau_{\varepsilon}(t), \bsj_{\varepsilon}(t))_{t \geq 0} - (\xi(t), \tau(t), \bsj(t))_{t \geq 0}  \Rightarrow 0$$ implies
$$
\mathbb{E}\left[ F \left( (\xi_{\varepsilon}(t), \tau_{\varepsilon}(t), \bsj_{\varepsilon}(t))_{t \geq 0} - (\xi(t), \tau(t), \bsj(t))_{t \geq 0}  \right)\right] \rightarrow 0,
$$
as $\varepsilon \downarrow 0$.
Now, the definition of the Wasserstein distance and the monotonicity of the integral imply for any $t \geq 0$:
\begin{align*}
d_{\mathrm{W}}(\mathrm{D}^{\varepsilon}_{t|0}(\cdot|\xi_0,j_0), \mathrm{D}_{t|0}(\cdot|\xi_0,j_0)) &\leq \mathbb{E}[\min\{1, \| \xi(t) - \xi_\varepsilon(t) \|^q\}] \\
&\leq \mathbb{E}\left[ F \left( (\xi_{\varepsilon}(t), \tau_{\varepsilon}(t), \bsj_{\varepsilon}(t))_{t \geq 0} - (\xi(t), \tau(t), \bsj(t))_{t \geq 0}  \right)\right] 
\end{align*}
Hence, we obtain the desired results by setting \\ $\alpha'(\varepsilon) := \mathbf{1}[\varepsilon > 0]  \mathbb{E}\left[ F \left( (\xi_{\varepsilon}(t)-\xi(t), \tau_{\varepsilon}(t)-\tau(t), \bsj_{\varepsilon}(t)-\bsj(t))_{t \geq 0} \right)\right]$.
\end{proof}

To prove the second part of this proposition, we proceed as follows:
we argue that the auxiliary process $(\xi_{\varepsilon}(t), \tau_{\varepsilon}(t), \bsj_{\varepsilon}(t))_{t \geq 0}$ behaves in its stationary regime like the SGPC setting with $\lambda := \mu(-\log(\varepsilon))$ in Lemma~\ref{Lemma_identical_stationar}. 
Then, however, we can show with Theorem~\ref{thm:ODE_limit}, that the process behaves like the \revised{full} gradient flow, as $\varepsilon \downarrow 0$. In Lemma~\ref{Lemma_gradient_flow_stationar}, we remind ourselves that the \revised{full} gradient flow has $\delta(\cdot - \theta^*)$ as a stationary measure. Finally, to prove Proposition~\ref{Propo_conv_auxiliary_pro}(ii) it will suffice to show that in Theorem~\ref{thm:ODE_limit}, also the corresponding stationary measures converge weakly.

\begin{lemma} \label{Lemma_identical_stationar} Let Assumptions~\ref{Ass_continuity} and \ref{Ass_Contr}(ii) hold. Moreover, let $\lambda := \mu(-\log(\varepsilon))$, let $\pi_C$ be the stationary distribution of $(\theta(t), \bsi(t))_{t \geq 0}$, and let $\pi_{\varepsilon}$ be the stationary distribution of $(\xi_{\varepsilon}(t), \tau_{\varepsilon}(t), \bsj_{\varepsilon}(t))_{t \geq 0}.$
Then, $$\pi_{C}(A \times J) = \pi_{\varepsilon}(A \times \{\varepsilon \} \times J), $$
for any $A \in \mathcal{B}X$ and $J \subseteq I$.
\end{lemma}
\begin{proof}
Note that the stationary measure of the process $(\xi_{\varepsilon}(t), \tau_{\varepsilon}(t), \bsj_{\varepsilon}(t))_{t \geq 0}$ does not change, when setting $\tau_{\varepsilon}(0) := \varepsilon$. Then however, $(\xi_{\varepsilon}(t),  \bsj_{\varepsilon}(t))_{t \geq 0}$ and $(\theta(t), \bsi(t))_{t \geq 0}$ are identically generated. Hence, they have the same stationary distribution. Also, Theorem~\ref{Thm_Hairer} and Theorem~\ref{Thm_Hairer2} imply that those stationary distributions are unique.
\end{proof}

\begin{lemma}  \label{Lemma_gradient_flow_stationar} Let Assumptions~\ref{Ass_continuity} and \ref{Ass_Contr}(ii) hold. Then, $\hPhi$ is strongly convex and for the flow $\bar{\varphi}$ corresponding to $\nabla \hPhi$, we have 
\begin{equation*}
\|\bar{\varphi}(\theta_0, t) -\bar{\varphi}(\theta_0', t) \| \leq \exp(-\kappa_1 t) \|\theta_0 - \theta_0' \|,
\end{equation*} 
where $\theta_0, \theta_0' \in X, t \geq 0$.
Hence, $\delta( \cdot - \theta^*)$ is the unique stationary measure of the \revised{full} gradient flow defined in \eqref{eq:zeta_ode}.
\end{lemma}
\begin{proof}
The first part follows from Lemma~\ref{Lemma_convex}. The second part is implied by the Banach Fixed-Point Theorem and by the stationarity of $\theta^*$ with respect to $\nabla \hPhi$.
\end{proof}

Now, we proceed to prove the second part of the main proposition. 
\begin{proof}[{Proof of Proposition~\ref{Propo_conv_auxiliary_pro}(ii)}] By Lemma~\ref{Lemma_identical_stationar} and Lemma~\ref{Lemma_gradient_flow_stationar}, it will be sufficient to show that in the setting of Theorem~\ref{thm:ODE_limit}, the stationary measure of SGPC with $\lambda := \mu(-\log(\varepsilon))$ converges to the stationary measure of the gradient flow $(\zeta(t))_{t \geq 0}$.
We proceed as in Chapters 6.4 and 6.5 of \cite{Kushner1984}, i.e. we need to show
\begin{itemize}
\item[(i)] $(\zeta(t))_{t \geq 0}$ has a unique stationary measure $\bar{\pi}$ and $\zeta(t) \Rightarrow \bar{\pi}$, as $t \rightarrow \infty$,
\item[(ii)] $\theta^*$ is Lyapunov stable for $(\zeta(t))_{t \geq 0}$,
\item[(iii)] Let $t_\varepsilon \rightarrow t_0 \in \mathbb{R}$, as $\varepsilon \downarrow 0$. Then, $\mathbb{P}(\theta(t_{\varepsilon}) \in \cdot) \Rightarrow \mathbb{P}(\zeta(0) \in \cdot)$, as $\varepsilon \downarrow 0$, implies that $(\theta(t_{\varepsilon} + t))_{t \geq 0} \Rightarrow (\zeta(t))_{t \geq 0}$, as $\varepsilon \downarrow 0,$
\item[(iv)] There is an $\varepsilon' > 0$, such that $(\theta(t))_{t \geq 0, \varepsilon' \geq \varepsilon > 0}$ is tight with respect to both $t$ and $\varepsilon$.
 \end{itemize}
Those assumptions will imply that $\theta(t) \Rightarrow \bar{\pi}$, as $\varepsilon \downarrow 0$ and $t \rightarrow \infty$; see Theorem 6.5 in \cite{Kushner1984}. As $(\theta(t))_{t \geq 0}$ has a unique stationary measure, we have that $\pi_{C} \Rightarrow \bar{\pi}$.
Now to prove these four assertions. (i), (ii) follow immediately from Lemma~\ref{Lemma_gradient_flow_stationar}, with $\bar{\pi} := \delta(\cdot - \theta^*)$. (iii) is implied by Theorem~\ref{Thm_kushner}. Due to the strong convexity that we have assumed in Assumption~\ref{Ass_Contr}(ii), we know that the process cannot \revised{escape} a certain compact set; see Lemma 1.14 in \cite{benaim2012:quant} for details. This implies tightness as needed in (iv).

Finally, note that $\pi_C \Rightarrow \bar{\pi}$ already implies that they also converge in $d_{W}$. Hence, we can construct a function $\alpha''$ accordingly.
\end{proof}

\subsection{Linear least squares problems} \label{Subs_LLSQ}
In this section, we illustrate the theoretical results of \S\S\ref{Subsec_const}--\ref{Appendix_Prop_approx} with an abstract example. 
In particular, we  show that Assumptions~\ref{Ass_continuity} and \ref{Ass_Contr} hold for linear least squares problems under weak assumptions.
Those appear in (regularised) linear or polynomial regression.

Let $Y:= \mathbb{R}^M$, $y \in Y$, and $G: X \rightarrow Y$ be a linear operator. $Y$ is the \emph{data space}, $y$ is the \emph{observed data set}, and $G$ is the \emph{parameter-to-data map}.
We consider the problem of estimating
\begin{equation} \label{eq_optprob_linearquad}
 \theta^* \in   \mathrm{argmin}_{\theta \in X} \hPhi(\theta) := \frac{1}{2}\|G \theta - y\|^2,
\end{equation}
which is called \emph{linear least squares problem}. 

We aim to solve this problem by the stochastic gradient descent algorithm. Indeed, we define
$$\Phi_i(\theta_0) := \frac12\|G_i \theta_0 - y_i\|^2 \qquad (\theta_0 \in X, i \in I),$$
where $y_i$ is an element of another Euclidean vector space $Y_i := \mathbb{R}^{M_i}$ and $G_i: X \rightarrow Y_i$ is a linear operator, for $i \in I$. We assume that these are given such that the space $Y = \prod_{i \in I}Y_i$, the vector $(y_i)_{i \in I} = N\cdot y$, and the operator $[G_1^T, \ldots, G_N^T]^T = N \cdot G.$
To define the SGP, we now need to derive the gradient field. This is given by the associated normal equations:
$$
\nabla \Phi_i(\theta_0) = G_i^TG_i \theta_0 - G_i^Ty_i \qquad (\theta_0 \in X, i \in I).
$$
These vector fields are linear, thus, satisfy Assumption~\ref{Ass_continuity}. Now we discuss Assumption~\ref{Ass_Contr}.
Let $i \in I$. Note that $G_i^TG_i$ is symmetric, positive semi-definite. We have
\begin{align*}
\langle \theta_0 - \theta_0', \Phi_i(\theta_0) -\Phi_i(\theta_0') \rangle = \langle \theta_0 - \theta_0', G_i^TG_i (\theta_0 -\theta_0') \rangle  \geq \kappa_i \|\theta_0 - \theta_0'\|^2,
\end{align*}
where $\kappa_i \geq 0$ is the smallest eigenvalue of  $G_i^TG_i$.
This implies that Assumption~\ref{Ass_Contr}(i) holds, if there is some $i \in I$ with $G_i^TG_i$ strictly positive definite. Furthermore, Assumption~\ref{Ass_Contr}(ii) holds, if for all $i \in I$ the matrix $G_i^TG_i$ is strictly positive definite.

Strict positive definiteness of $G_i^TG_i$ is satisfied, if $\dim Y_i \geq \dim X$ and $G_i$ has full rank, for $i \in I$. 
The inequality $\dim Y_i \geq \dim X$ is not restrictive, as we apply SGD typically in settings with very large data sets. If the $G_i$ do not have full rank, one could add a Tikhonov regulariser to the target function in \eqref{eq_optprob_linearquad}.

\section{From continuous to discrete} \label{Sec_Discret}
In the previous sections, we have introduced and discussed SGP mainly as an analytical tool and abstract framework to study SGD. 
However, we can also apply SGP more immediately in practice.
To this end, we need to consider the following computational tasks:
\begin{itemize}
\item[(i)] discretisation of deterministic flows $(\varphi_i)_{i \in I}$
\item[(ii)] discretisation of continuous-time Markov processes $(\bsi(t))_{t \geq 0}$, resp. $(\bsj(t))_{t \geq 0}$
\end{itemize}
The discretisation of the $(\varphi_i)_{i \in I}$ consists in the discretisation of several homogeneous ODEs. The discretisation of ODEs has been studied extensively; see, e.g., \cite{iserles2008}.  Thus, we focus on (ii) and discuss a sampling strategy for the CTMPs in \S\ref{Subsec_Applic_SGP}.

A different aspect is the following: note that when specifying strategies for (i) and (ii), we implicitly construct a stochastic optimisation algorithm. Since we have introduced SGP as a continuous-time variant of SGD, one of these algorithms should be the original SGD algorithm.  Indeed, in \S\ref{Subsec_retrieve_SGD}  we will explain a rather crude discretisation scheme which allows us to retrieve SGD. Well-known algorithms beyond SGD that can be retrieved from SGP are discussed in \S \ref{Subsec_beyond_SGD}.
\subsection{Applying SGP} \label{Subsec_Applic_SGP}
We now briefly explain a  strategy that allows us to sample the CTMPs $(\bsi(t))_{t \geq 0}$ and  $(\bsj(t))_{t \geq 0}$. Without loss of generality, we focus on the second case, $(\bsj(t))_{t \geq 0}$.

Indeed, we give a sampling strategy in Algorithm~\ref{alg_bsj}. It commences by sampling an initial value $\bsj(0)$. 
This value remains constant for the duration of the random waiting time.
After this waiting time is over, we sample the next value of the process from a uniform distribution  on all states, but the current state. 
This value is kept constant for another random waiting time and so on.
This strategy goes back to Gillespie \cite{Gillespie1977}; see also \cite{Rao2012} for this and other sampling strategies for CTMPs on discrete spaces.
\begin{algorithm}[htbp]
\caption{Sampling $(\bsj(t))_{t \geq 0}$}\label{alg_bsj} 
\begin{algorithmic}[1] \State sample $\bsj(0) \sim \mathrm{Unif}(I)$
\State $T_0 \gets 0$
\For{$k = 1, 2, \ldots$}
      \State sample $D \sim \pi_{\rm wt}(\cdot | T_{k-1})$ \label{line_sampling_wt}
      \State $T_k \gets T_{k-1}+ D$
      \State $\bsj|_{[T_{k-1}, T_k)} \gets \bsj(T_{k-1}) $
      \State  $\bsj(T_{k}) \sim \mathrm{Unif}(I \backslash \{\bsj(T_{k-1})\})$
   \EndFor
   \State \textbf{return} $(\bsj(t))_{t \geq 0}$
\end{algorithmic}
\end{algorithm}

The potentially most challenging step in Algorithm~\ref{alg_bsj} is the sampling from the  distribution $\pi_{\rm wt}(\cdot | t_0)$ in line~\ref{line_sampling_wt}. 
In the case of SGPC, i.e. if $\eta$ is constant, this sampling just comes down to sampling from an exponential distribution. 
In SGPD, the sampling could be performed using the quantile function of $\pi_{\rm wt}(\cdot | t_0)$, if accessible. 
We sketch the  method below. If the quantile function is not accessible, strategies such as rejection sampling may be applicable; see \cite{Robert2004} for details. 
In the following, we consider \revised{first} the case where $1/\eta(\cdot)$ is an affine function \revised{and then the case where $\eta$ scales exponentially in time. Both of these cases satisfy the growth condition in \eqref{Eq_bounded_deriv}. Thus, our theory applies to the SGPD employing either of these learning rate functions.}
\begin{example}\label{Exam_linear_fail}
Let $\eta(t) := (at +b)^{-1}$, for $t\geq 0$ and some $a, b > 0$. Then, we have for $t_0 \geq 0$ and $t \geq t_0$:
\begin{align*}
\pi_{\rm wt}((-\infty, t]| t_0) =1- \exp\left(-\int_{0}^t au+at_0+b \mathrm{d}u \right) = 1-\exp\left(-\frac{1}{2}at^2 - at_0t - bt \right).
\end{align*}
By inverting this formula, we obtain the quantile function
\begin{equation} \label{eq_quantile_func}
Q(s|t_0) = \frac{-at_0 - b + \sqrt{(at_0+b)^2 - 2a\log(1-s)}}{a},
\end{equation}
where $s \in (0,1), t_0 \geq 0$.
Using this quantile function, we are able to sample from $\pi_{\rm wt}( \cdot | t_0)$. Note that for $\revised{U} \sim \mathrm{Unif}((0,1))$ we have $\mathbb{P}(Q(\revised{U}|t_0) \in \cdot) = \pi_{\rm wt}(\cdot| t_0)$.
We have used this technique to estimate mean and standard deviations of $\pi_{\rm wt}( \cdot | t_0)$ for $t_0 \in [0, \revised{10}]$ and $a = b = 1$; see Figure~\ref{fig_failure_dist}. We observe that the mean behaves as $\eta(\cdot)$, showing a similarity with the exponential distribution.
\end{example}
\begin{figure}[htb]
\centering
\includegraphics[scale=0.5]{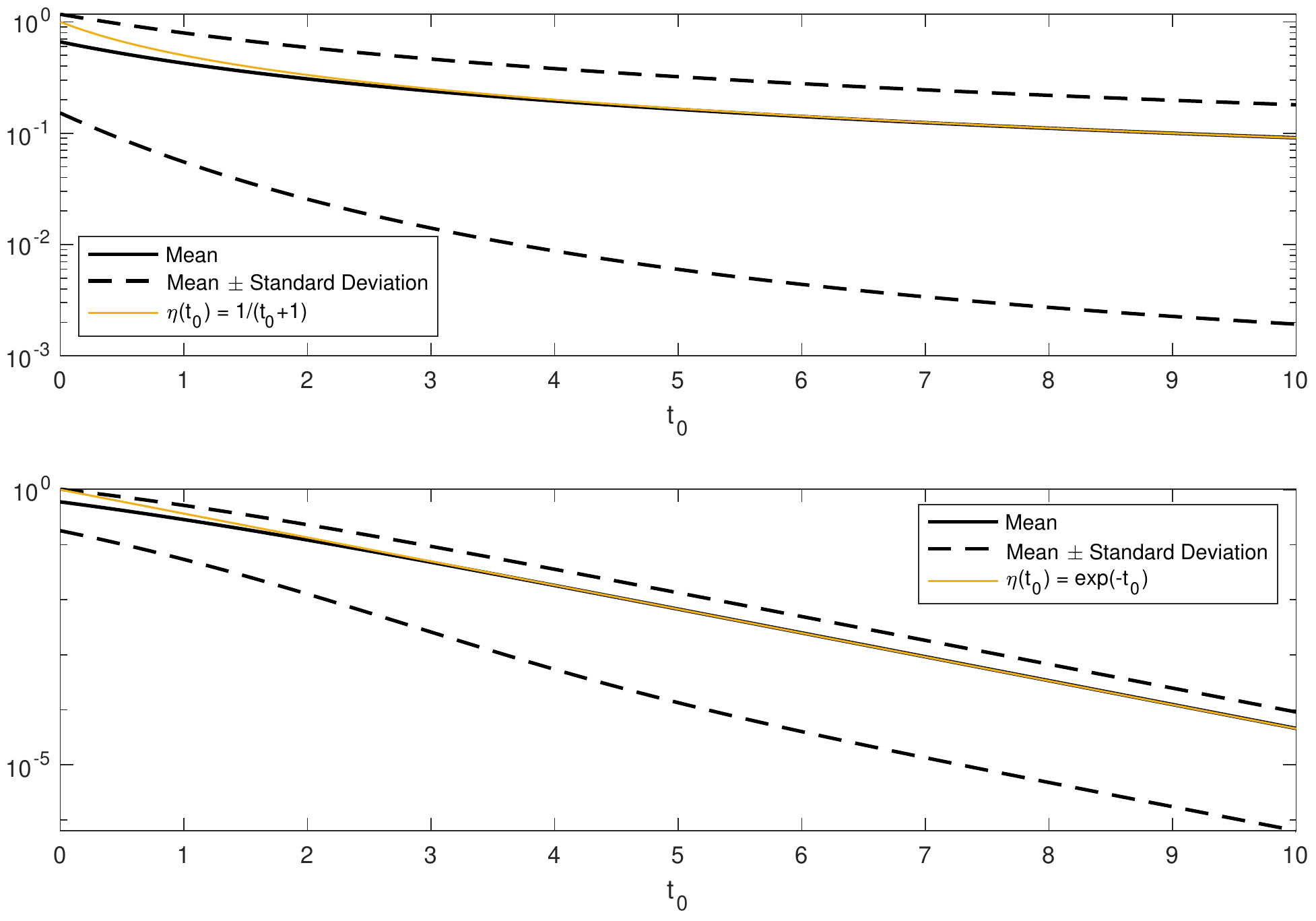}
\caption{Mean and standard deviations for the time-dependent probability measure\revised{s} $\pi_{\rm wt}(\cdot| t_0)$ from Example\revised{s}~\ref{Exam_linear_fail} \revised{(top row) and \ref{Exam_exponential_fail} (bottom row)} with $a = b = 1$ and $t_0 \in [0, \revised{10}]$. Mean and standard deviations are estimated with standard Monte Carlo using $10^4$ samples.}
\label{fig_failure_dist}
\end{figure}
\revised{
\begin{example}\label{Exam_exponential_fail}
Let $\eta(t) := a\exp(-bt)$, for $t\geq 0$ and some $a, b > 0$. Then, we have for $t_0 \geq 0$ and $t \geq t_0$:
\begin{align*}
\pi_{\rm wt}((-\infty, t]| t_0) =1- \exp\left(-\int_{0}^t \frac{\exp(b(u+t_0))}{a} \mathrm{d}u \right) = 1-\exp\left(\frac{1-\exp(bt)}{ab\exp(-bt_0)} \right).
\end{align*}
We can again compute the quantile function
\begin{equation} \label{eq_quantile_func_exp}
Q(s|t_0) = \frac{1}{b} \log\left(1- ab\exp(-bt_0)\log(1-s)\right)
\end{equation}
where $s \in (0,1), t_0 \geq 0$. We again use the quantile function to estimate mean and standard deviations of the distribution for  $a = b = 1$ and $t_0 \in [0,10]$; see Figure~\ref{fig_failure_dist}. 
\end{example}}

\subsection{Retrieving SGD from SGP} \label{Subsec_retrieve_SGD} Now, we discuss how the SGP dynamic needs to be discretised to retrieve the SGD algorithm. To this end, we list some features that we need to keep in mind:
The waiting times between switches of the data sets are deterministic in SGD and random in SGP. The processes  $(\bsi(t))_{t \geq 0}$ and  $(\bsj(t))_{t \geq 0}$ in SGP indeed jump with probability one after the waiting time is over, i.e. $\bsi(t) \neq \bsi(s)$ when one jump occurred in $(t,s]$. In SGD, however, it is possible to have a data set picked from the sample twice in a row.
Finally, we need to discretise the flows $(\varphi_i)_{i \in I}$ using the explicit Euler method.

We approximate the process $(\bsj(t))_{t \geq 0}$ by
\begin{equation} \label{eq_jhat_sgd}
\hat{\bsj}(t) := \sum_{k=0}^{\infty} \bsj_k \mathbf{1}\left[\hat{t}_{k} \leq t < \hat{t}_{k+1}\right],
\end{equation}
where $\bsj_0,\bsj_1,\ldots \sim \mathrm{Unif}(I)$ i.i.d. and the sequence $(\hat{t}_k)_{k=0}^\infty$ is given by
\begin{equation} \label{eq_wt_cont2disc}
\hat{t}_0 := 0, \quad \hat{\eta}_{k+1} := \eta\left(\hat{t}_k\right), \quad  \hat{t}_k := \sum_{\ell = 1}^k\hat{\eta}_\ell \quad (k \in \mathbb{N}).
\end{equation}
Note that with this definition of the sequence $(\hat{\eta}_k)_{k=1}^\infty$ , we obtain  $\hat{\eta}_k = \eta_k$, $k \in \mathbb{N}$, which was the discrete learning rate defined in Algorithm~\ref{alg_SGD}.
See our discussion in \S \ref{Subsec_Choice_of_model} for the choice of $(\hat{\bsj}(t) )_{t \geq 0}$ as an approximation of $({\bsj}(t) )_{t \geq 0}$.
If we employ $(\hat{\bsj}(t))_{t \geq 0}$ and an explicit Euler discretisation with step length $\eta_k$ in step $k \in \mathbb{N}$ to discretise the respective flows $(\varphi_i)_{i \in I}$, we obtain precisely the process defined in Algorithm~\ref{alg_SGD}.

\subsection{Beyond SGD} \label{Subsec_beyond_SGD}
In \S\ref{Subsec_retrieve_SGD}, we have discussed how to discretise the SGP $(\xi(t))_{t \geq 0}$ to obtain the standard SGD algorithm. It is also possible to retrieve other stochastic optimisation algorithms by employing other discretisation strategies for the flows $(\varphi_i)_{i \in I}$. 
Note, e.g., that when replacing the explicit Euler discretisation of the flows $(\varphi_i)_{i \in I}$ in \S\ref{Subsec_retrieve_SGD}  by an implicit Euler discretisation, we obtain the \emph{stochastic proximal point algorithm}; see, e.g., Proposition 1 of \cite{Bertsekas2011} for details. 

Using higher-order methods instead of explicit/im\-plicit Euler, we obtain higher-order stochastic optimisation methods. Those have been discussed by Song et al. \cite{Song2018}. Adaptive Learning Rates for SGD are conceptually similar to adaptive stepsize algorithms in ODE solvers, but follow different ideas in practice; see \cite{Duchi2011,Li2019}.

Linear-complexity SGD-type methods, like Stochastic Average Gradient (SAG) \cite{Schmidt2017}, Stochastic Variance Reduced Gradient (SVRG) \cite{Johnson2013}, or SAGA \cite{Defazio2014}  remind us of multistep integrators for ODEs. Here, the update does not only depend on the current state of the system, but also on past states. On the other hand, variance reduction in the discretisation of stochastic dynamical systems is, e.g.,  the object of Multilevel Monte Carlo path sampling, as proposed by Giles \cite{Giles2008}.


\section{Numerical \revised{experiments}} \label{Sec_NumIll}
We now aim to get an intuition behind the stationary measures $\pi_C$, $\pi_\varepsilon$ (Theorems~\ref{Thm_Hairer} and \ref{Thm_Hairer2}), \revised{study} the convergence of the Markov processes, and \revised{compare } SGP with SGD.
\begin{figure*}[htb]
\centering
\includegraphics[scale=0.9]{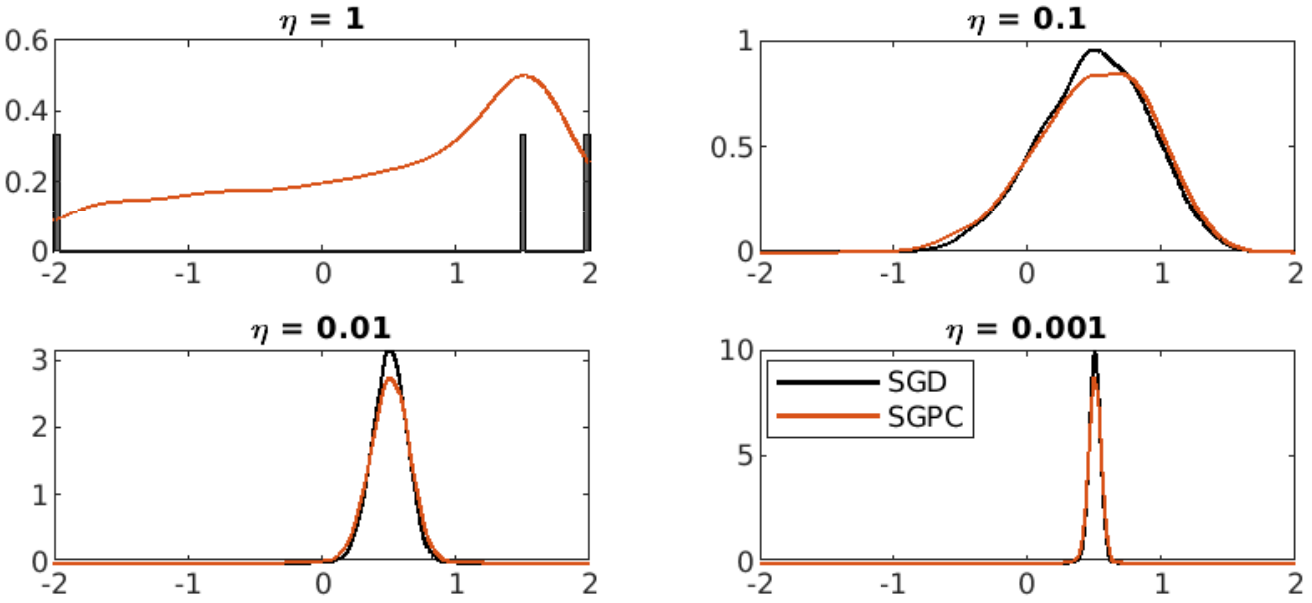}
\caption{Estimated stationary measures of SGD and SGPC with different $\eta \in  \{1,  10^{-1}, 10^{-2}, 10^{-3}\}$ and initial value $\theta_0 = -1.5$. The results are based on kernel density estimations with $10^4$ samples each of $\theta(10)$ for SGPC and $\theta_k$ with $k = 10/\eta$ for SGD. \revised{Note that for SGD with $\eta = 1$, the samples are concentrated in 3 points, which is why we plot a histogram rather than a density.}}
\label{Fig_toy_example}
\end{figure*}

Below, we define the academic example that we study throughout this section. It fits into the linear least squares framework discussed in \S\ref{Subs_LLSQ}. Moreover, it  satisfies Assumptions~\ref{Ass_continuity} and \ref{Ass_Contr}\revised{(i) and }(ii); see \S \ref{Subs_LLSQ}. Then, we proceed by applying SGD, SGPC, and SGPD.
\begin{example} \label{Exam_toy}
Let $N := 3$, i.e. $I:=\{1, 2, 3\}$, and  $X := \mathbb{R}$. We define the potentials 
\begin{align*}
\Phi_1(\theta) :=\frac{1}{2}(\theta+2)^2, \quad \Phi_2(\theta) :=\frac{1}{2}(\theta-1.5)^2, \quad  \Phi_3(\theta) :=\frac{1}{2}(\theta-2)^2 \quad (\theta \in X).
\end{align*}
 The minimiser of $\bar{\Phi} \equiv \Phi_1/3 + \Phi_2/3 + \Phi_3/3$ is $\theta^* = 0.5$. 
\end{example}
\subsection{Constant learning rate} \label{Numerics_constant_eta}
Approaching the optimisation problem in Example~\ref{Exam_toy}, we now employ SGPC with initial value $\theta_0 = -1.5$ and $\eta \in  \{1, 10^{-1}, 10^{-2}, 10^{-3}\}$. We sample from this process using Algorithm~\ref{alg_bsj} for the CTMP $(\bsi(t))_{t \geq 0}$  and \revised{the analytical solution of the gradient flows} $(\varphi_i)_{i \in I}$.  
Throughout this section, we use the \textsc{Matlab} function \texttt{ksdensity} to compute kernel density estimates. All of those are based on Gaussian kernel functions with boundary correction at $\{-2, 2\}$, if necessary.

We now sample SGPC as discussed above and collect the samples $\theta(10)$, i.e. the value of the process at time $t = 10$.  
In Figure~\ref{Fig_toy_example}, we show kernel density estimates based on $10^4$ of these samples. 
For large $\eta$, the density has mass all over the invariant set of the $(\varphi_i)_{i \in I}$. 
If $\eta$ is reduced, we see that the densities become more and more concentrated around the optimum $\theta^*$.

Next, we compare SGPC with SGD. Indeed, we compute kernel density estimates of $10^4$ samples of the associated SGD outputs. In particular, we run SGD with the same learning rates up to iterate $10/\eta$.  For $\eta = 1$, the numerical artifacts seem to dominate SGD. 
 For smaller $\eta$, the densities obtained from both algorithms behave very similarly: we only see a slightly larger variance in SGP. Indeed, when looking at the values of the variances of $\theta(10)$ for $\eta \in  \{10^{-1}, 10^{-2}, 10^{-3}\}$, they seem to depend linearly on $\eta$ and only differ among each other by about factor 1.3, see the estimates in Table~\ref{table_SGD_var}.
\begin{table}[htb]
\centering
\begin{tabular}{l|llll} \hline
$\eta$ & 1  & $10^{-1}$  & $10^{-2}$ & $10^{-3}$  \\ \hline
SGPC & 1.2741 & 0.1961 & 0.0209  & 0.0021   \\
SGD     & 3.1754 &   0.1695  &  0.0157  &  0.0016  \\ \hline
\end{tabular}
\caption{Sample variances of $10^4$ samples of $\theta(10)$ in SGPC and $\theta_{10/\eta}$ in SGD.}
\label{table_SGD_var}
\end{table}

We next take a look at the sample paths of said SGPC runs; consider Figure~\ref{Fig_samplepaths}. As anticipated and actually already shown in Figure~\ref{Fig_odelim}, the smaller $\eta$ leads to a faster switching and to a sample path that well approximates the full gradient flow. Large $\eta$ leads to slow switching. 
It is difficult to recognise the actual speed of convergence shown in Theorem~\ref{Thm_Hairer}.
However, we see that each of the chains indeed reaches a stationary regime. The time at which those regimes are reached highly depends on $\eta$. Indeed, for $\eta = 1$ we seem to be almost right away in said regime. For the smallest learning rate $\eta = 10^{-3}$, it appears to take up to $t \approx 3.5$. 
What does this mean from a computational point of view? The approach with a small learning rate is computationally inefficient: the large number of switches makes the discretisation of the sample paths computationally expensive; the slow convergence to the stationary regime implies that we need to run the process for a relatively long time. 
For large $\eta$, however, we are not able to identify the optimal point; see Figure~\ref{Fig_toy_example}. Hence, with large and constant $\eta$ the method is ineffective.

\begin{figure}[htb]
\centering
\includegraphics[scale=0.65]{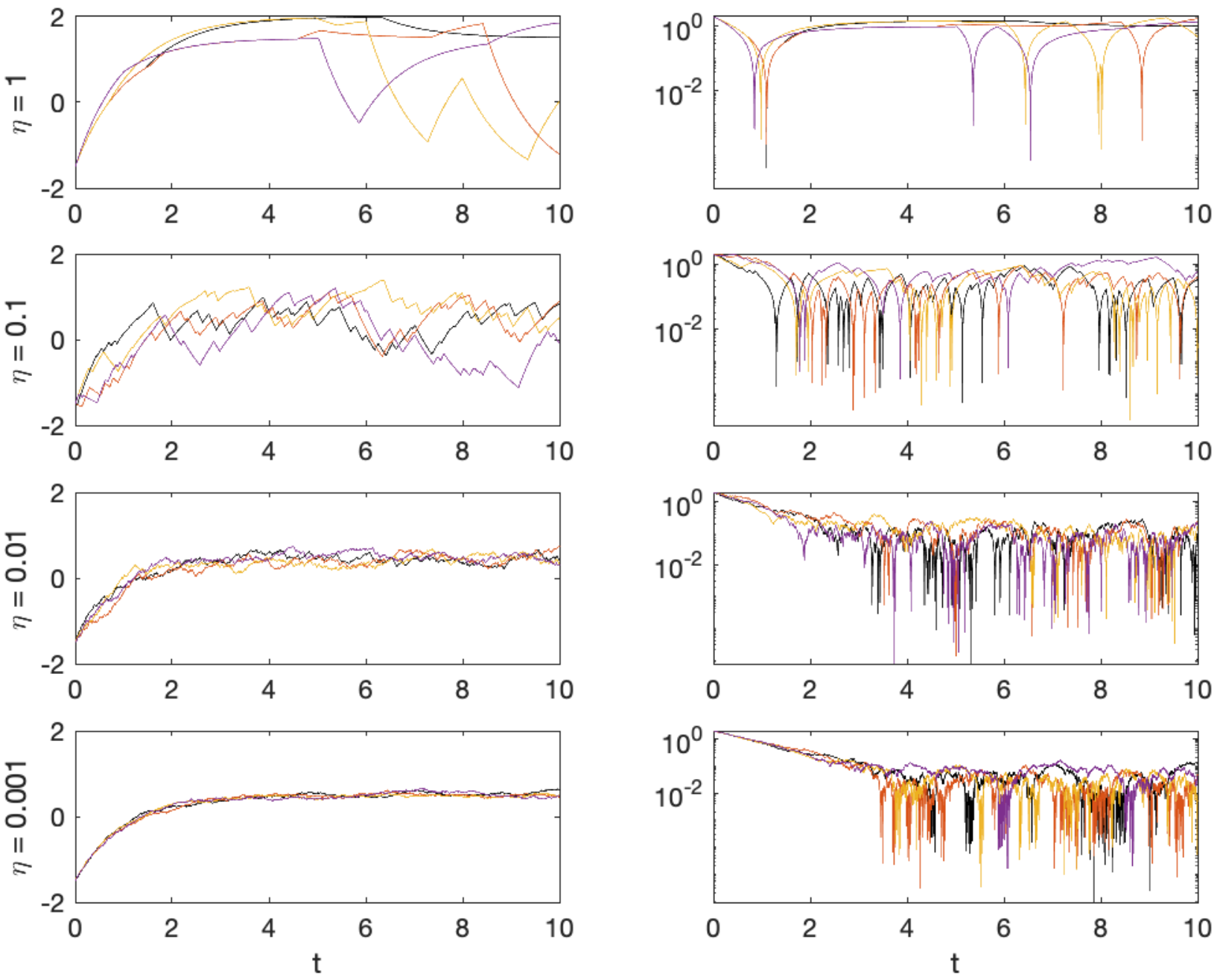}

\caption{Sample paths of SGPC  as in Figure~\ref{Fig_toy_example}. Left: four sample paths $(\theta(t))_{t \geq 0}$, right: associated distances between sample paths and optimal point, i.e. $(|\theta(t)-0.5|)_{t \geq 0}$.}

\label{Fig_samplepaths}
\end{figure}
\subsection{Decreasing learning rate} \label{Numerics_decreasing_eta}
In SGPD, we can solve the efficiency problem of SGPC noted in the end of \S\ref{Numerics_constant_eta}: we start with a large $\eta$, which is decreased over time. Hence, we should expect to see fast convergence in the beginning and accurate estimation of $\theta^*$ later on. 
To test this assertion we get back to the problem defined in Example~\ref{Exam_toy}.

\begin{figure}[htb]
\centering
\includegraphics[scale=0.5]{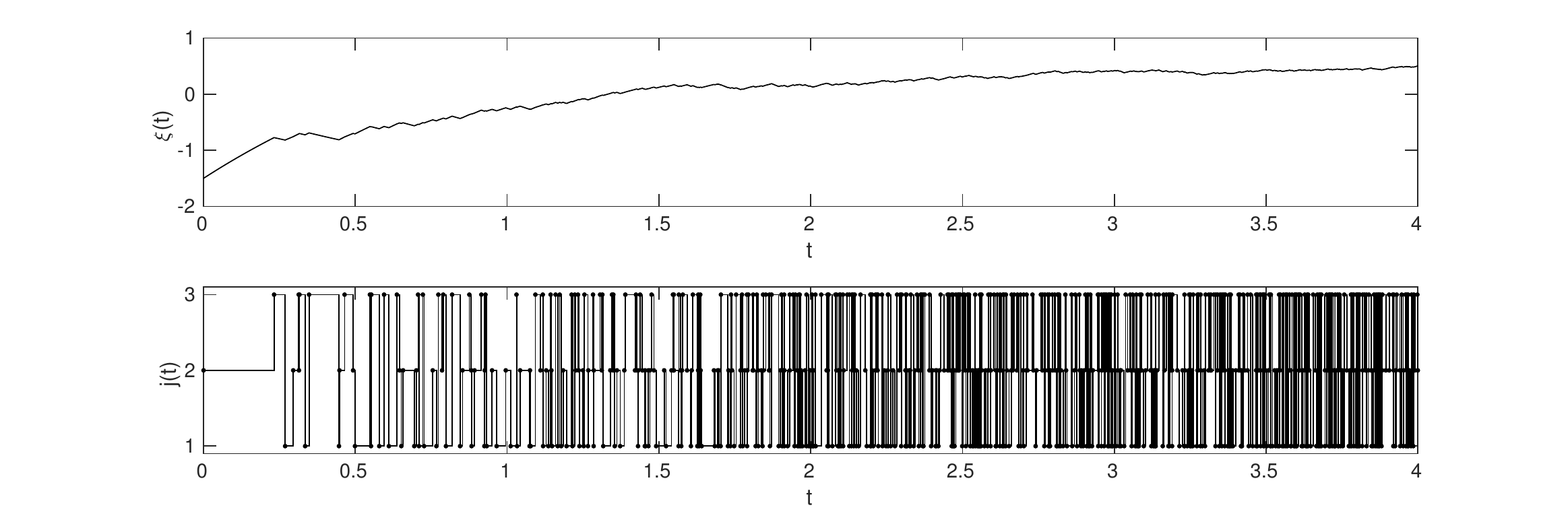} \\

\includegraphics[scale=0.5]{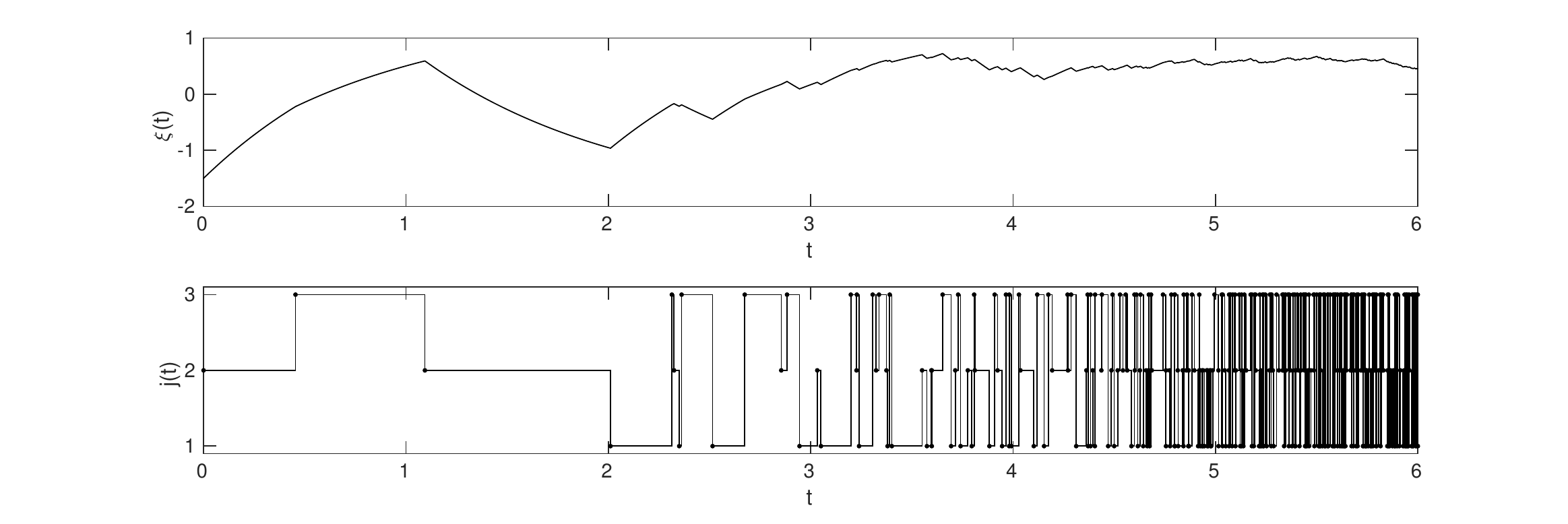}
\caption{A sample path of $(\xi(t), \bsj(t))_{t \geq 0}$, as specified in \S\ref{Numerics_decreasing_eta}. \revised{The top two figures refer to the rational learning rate, the bottom two figures refer to the exponential rate.}}
\label{Figure_paths_i_xi}
\end{figure}

\begin{figure*}[htp]
\centering
\includegraphics[scale=0.6]{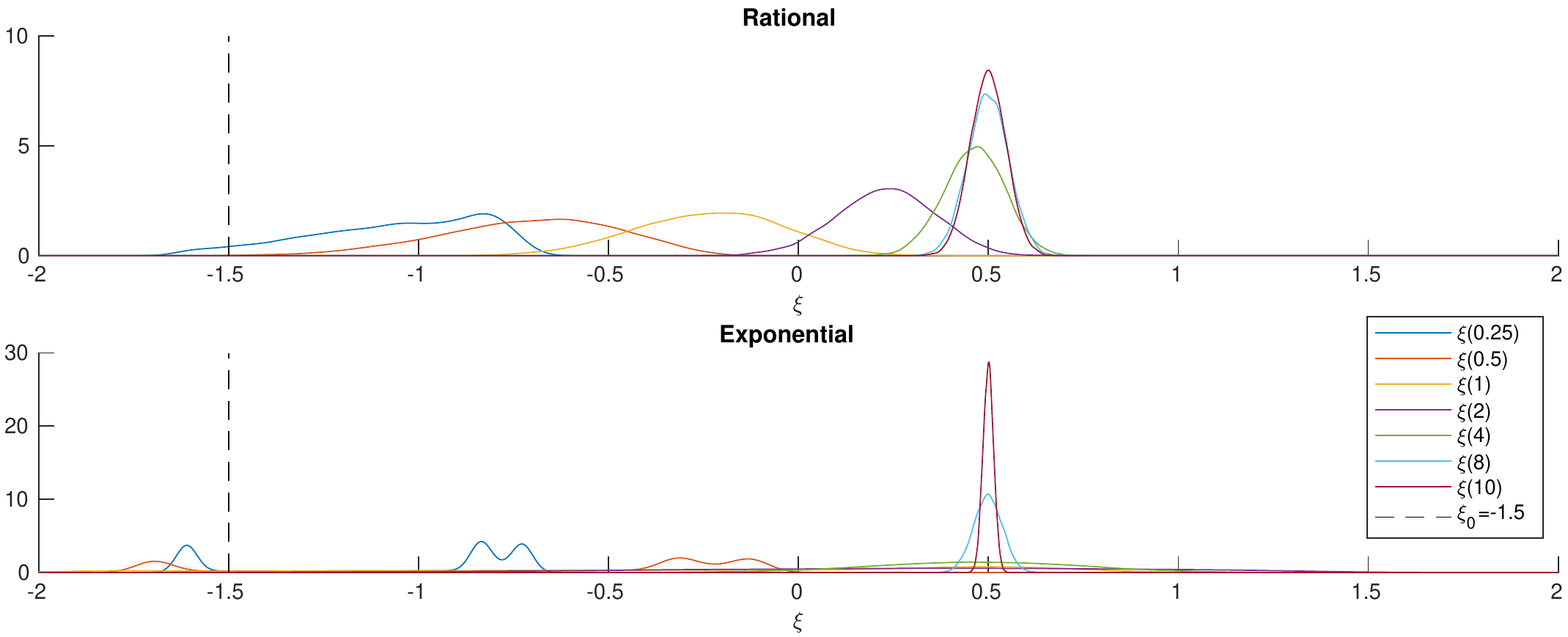}
\caption{Estimated densities of the distribution of the SGPD states using $10^4$ Monte Carlo samples. 
Densities at times $t \in \{1/4, 1/2, 1, 2, 4, \revised{8,} 10\}$ and initial value $\xi(0) = \xi_0$. 
}
\label{Fig_SGPD_KDE}
\end{figure*}

We \revised{study two different} time-dependent learning rate\revised{s: a rational rate that is the reciprocal of an affine function, as in } Example~\ref{Exam_linear_fail}\revised{, as well as an exponential learning rate; as in Example~\ref{Exam_exponential_fail}.} In particular, we choose
\begin{align*}
\eta(t) &:= \frac{1}{100t +1} \tag{rational} \\
\revised{\eta(t)} &:= \exp(-t)\tag{exponential}.
\end{align*}
and sample from the associated waiting time distribution using the quantile function\revised{s} \eqref{eq_quantile_func} \revised{ and \eqref{eq_quantile_func_exp}, respectively. Note that, as mentioned before, the reciprocal of both learning rate functions satisfies the growth condition in \eqref{Eq_bounded_deriv}.} 
All the other specifications are identical to the ones given in \S\ref{Numerics_constant_eta}: we set, e.g., $\xi_0 := -1.5$ as an initial value for the process. 
In Figure~\ref{Figure_paths_i_xi}, we show single sample path\revised{s}  of the process\revised{es}  $(\xi(t), \bsj(t))_{t \geq 0}$, \revised{with the different learning rate functions}. 
\revised{In both cases,} we can see that the waiting times between jumps in $(\bsj(t))_{t \geq 0}$ go down as $t$ increases: the (vertical) jumps become denser over time. For small $t > 0$, one can also recognise the coupling between $(\xi(t))_{t \geq0}$ and $(\bsj(t))_{t \geq 0}$. \revised{If we compare the paths with the different learning rate functions, we see that the exponential rate allows for much larger steps in the beginning and then decreases quite quickly. The rational rate leads to fast switching early on, which decreases further rather slowly over time. }
Note that \revised{these} plot\revised{s are} essentially realistic version\revised{s} of the cartoon in Figure~\ref{Fig_pdmp_cartoon}.

Next, we look at the distribution of $\xi(t)$ for particular $t > 0$. In Figure~\ref{Fig_SGPD_KDE}, we plot kernel density estimates for the distributions of $\xi(1/4),$ $\xi(1/2),$ $\xi(1),$ $\xi(2),$ $\xi(4),$ \revised{$\xi(8)$} and $\xi(10)$. Those estimates are each based on $10^4$ independent Monte Carlo samples. Hence, we show how the distribution of the process\revised{es} evolves over time. We observe that the process starting at $\xi(0) = -1.5$ moves away from that state and 
slowly approaches the optimal point $\theta^* = 0.5$. Doing so, it starts with a large variance that is slowly decreased over time. This is consistent with what we have observed in Figure~\ref{Fig_toy_example} and Table~\ref{table_SGD_var}.  \revised{In case of the exponential learning rate, this behaviour is much more pronounced: we start with a much higher variance but end up at $t=10$ with a smaller variance.}

\revised{In Figure~\ref{Figure_comparison_densities}, we additionally compare the distribution of the constant learning rate process with $\eta = 10^{-3}$ with the exponential and rational rate processes at the time at which their learning rate is approximately equal to $10^{-3}$. We see that the states of the constant and rational rate processes have almost the same distribution, which is what we would hope to see. The exponential learning rate process has a larger variance.}

\begin{figure}
\centering
\includegraphics[scale=0.6]{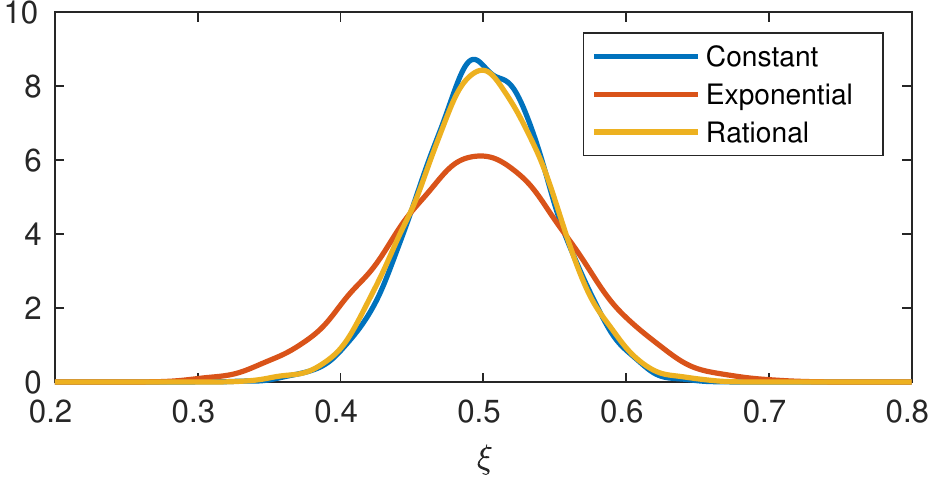}
\caption{\revised{Comparison of the densities of SGPC state $\theta(10)$ where $\eta = 10^{-3}$ taken from Figure~\ref{Fig_toy_example}, the rational learning rate SGPD $\xi(9.99)$, and the exponential learning rate SGPD $\xi(6.91)$. The densities are estimated with $10^4$ samples. }}
\label{Figure_comparison_densities}
\end{figure}

\begin{figure*}[htb]
\centering
\includegraphics[scale=0.7]{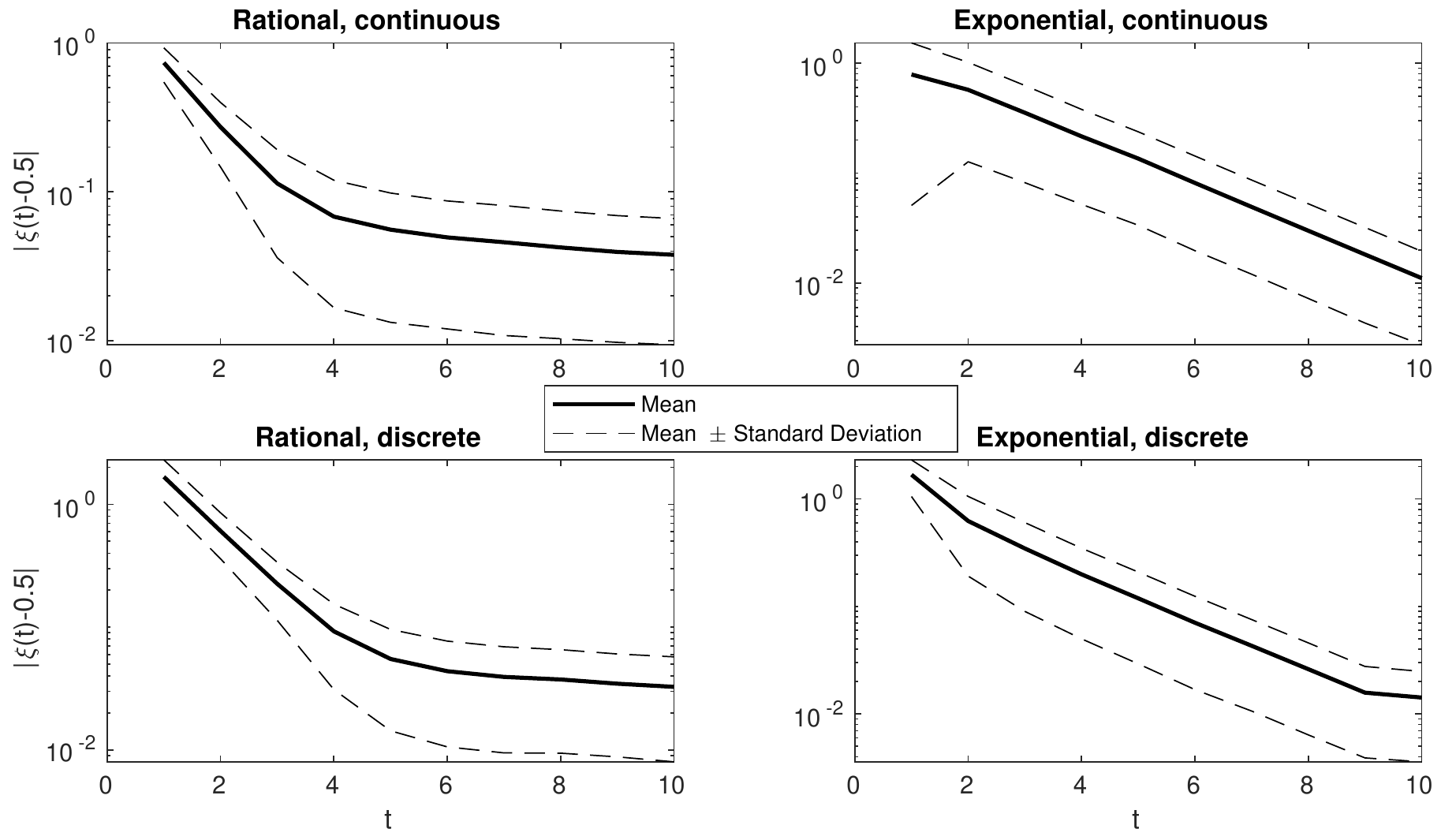}
\caption{\revised{Mean of the absolute error $|\xi(t)-0.5|$, estimated at $t = 1,2,\ldots,10$ with $10^4$ Monte Carlo samples and associated standard deviations.}}
\label{Figure_Error_MC_SGPD}
\end{figure*}

\revised{To study the performance of SGPD quantitatively, we estimate mean and standard deviation of the absolute error $|\xi(t)-0.5|$ at $t = 1,2,\ldots,10$ using $10^4$ Monte Carlo samples.   To see the full context, we also performed $10^4$ runs of the associated discrete-time SGD algorithms. The learning rate sequences $(\eta_k)_{k=1}^\infty$ are chosen as we have suggested in \eqref{eq_wt_cont2disc}. We show the results in Figure~\ref{Figure_Error_MC_SGPD}. In the exponential, continuous case, we see an exponential convergence rate. In all the other settings, the rates are sublinear.  For the discrete settings, this is exactly what we would expect based on the literature; see  \cite{Jentzen2018} and the references therein. Interestingly, the rational, continuous case appears to be less efficient than the rational, discrete case. This could imply that the learning rate function is supposed to be chosen according to the convergence rate of the underlying deterministic dynamical system. }

%


\section{Conclusions} \label{Sec_conclu}
We have proposed the stochastic gradient process as a natural continuum limit of the popular stochastic gradient descent algorithm. 
It arises when replacing the explicit Euler updates by the exact gradient flows and the waiting times between data switches by \revised{appropriate}  random waiting times.
This continuous-time model is a piecewise-deterministic Markov process. \revised{It represents the uniform subsampling from a finite set of potentials after strictly positive waiting times, the Markovian nature of SGD, the switching of potentials, and the approximation of the full deterministic gradient flow. Moreover, the process}  has an interpretation in population dynamics.

Within this continuum limit, we are able to study Wasserstein ergodicity in the case of strongly convex target functions. \revised{In the case of constant learning rates, we obtain exponential ergodicity. A similar result has been established by \cite{Dieuleveut2017} in discrete time. In the case of decreasing learning rates, we could show weak convergence to the minimiser of the target function. Our results do not allow us to assess the convergence rate in that case. 
Numerical experiments indicate that it depends on the underlying data switching process and could in certain cases be exponential as well.}

\revised{In the numerical experiments, we compared samples from SGP with samples from SGD. Here, we, for instance,} observed strong similarities \revised{between the} stationary measure of the \revised{two processes}. Indeed, we claim that our continuum limit is a good representation of stochastic gradient descent in the long-time limit.
\revised{Here, we have been able to sample accurately from SGP, as the flows attain analytical representations. In most practical cases, we would need to construct a discrete stochastic optimisation algorithm from SGP using an ODE integrator. Following this machinery, one can also retrieve known stochastic optimisation algorithms, showing that SGP is also a generalisation of those.}

We conclude this work with \revised{four} remarks. Here, we discuss possible extensions of the  stochastic gradient process framework.
\begin{remark}[Global, non-convex]
Throughout our long-time analysis, we have required strong convexity of the target functions. In practical applications, e.g. the training of deep neural networks, convexity is too strong. If certain H\"ormander bracket conditions are satisfied, exponential ergodicity may also be shown without the strong convexity assumption, see, e.g. \cite{Bakhtin2012,Cloez2015}. This does not yet imply that the processes will converge to the global optimum, if $\eta \downarrow 0$. However, we remark that the densities in the numerical illustrations in \S\ref{Sec_NumIll} very much remind us of a simulated annealing scheme, where $\eta$ controls the variance of the target measure; see e.g. \S 5.2.3 of \cite{Robert2004}. 
In some cases, simulated annealing is able to find global extrema of non-convex target functions; see \cite{Yang2000}. Hence, this connection may fertilise future research in this direction.
\end{remark}
\begin{remark}[Constrained]
SGD has been successfully applied in constrained optimisation; typically by projecting each update on the space of feasible vectors. This is difficult to represent in the SGP setting; as the projection would need to be part of the piecewise ODEs. However, PDMPs on bounded sets already appear in the original paper by Davis \cite{Davis84}. Here, a jump is introduced as soon as the boundary of the feasible set is reached. In SGP, one could introduce a jump in the continuous\revised{-}time Markov process $(\bsi(t))_{t \geq 0}$ and $(\bsj(t))_{t \geq 0}$, as soon as the boundary is hit. Hence, the data set is randomly switched until the process moves away from the boundary or the boundary point is stationary for the process. 
\end{remark}
\begin{remark}[Gradient-free]
In this work, we cover only methods that are fundamentally based on discretised gradient flows.  Other stochastic optimisation algorithms are based on other underlying dynamics. Such are ensem\-ble-based methods or evolutionary algorithms. Consider, for instance, the ensemble Kalman inversion framework, which was proposed by Schillings and Stuart \cite{Schillings2017} as a continuum limit of some ensemble Kalman filter. Using our SGP view, one may be able to analyse subsampling in ensemble Kalman inversion, as proposed by \cite{Kovachki2019}.
\end{remark}
\revised{
\begin{remark}[Non-Markovian] \label{Rem_nonMarkovian} We have modelled SGP as a piecewise-deterministic Markov process. In practice, one might be interested in non-Markovian extensions to this setting. Non-Markovian settings arise, e.g., when adapting the learning rate throughout the algorithm, as in the celebrated AdaGrad algorithm \cite{Duchi2011}. 

Another non-Markovian extension is the following. In the present work, we have decided to switch the potentials in the SGPs after random waiting times. While this allowed us to study SGP as a (piecewise-deterministic) Markov process, it did not retain SGD's property of jumping after deterministic waiting times. If we model the waiting times deterministically, the processes $(\bsi(t))_{t \geq 0}$, $(\bsj(t))_{t \geq 0}$ become general renewal processes and non-Markovian. Especially since deterministic waiting times are easier to handle in practice, the then resulting `renewal stochastic gradient processes' are highly interesting objects for future studies.
\end{remark}}

\revised{
\appendix
\section{Auxiliary results concerning CTMPs} \label{Appendix_CTMPs}
In this appendix, we give a brief derivation of the Markov kernel describing the processes $(\bsi(t))_{t \geq 0}$ and $(\bsj(t))_{t \geq 0}$. Moreover, we discuss the non-explosiveness of $(\bsj(t))_{t \geq 0}$, i.e. we show that the sequence of jump times $(T_k)_{k=1}^\infty$ satisfies $$\mathbb{P}\left(\lim_{k \rightarrow \infty} T_k = \infty\right) = 1.$$ We commence with the discussion of the Markov kernels.}
\revised{
\begin{lemma} \label{Lemma_Appendix_M}
Let $M_t: I \times 2^I \rightarrow[0,1]$ be given by \begin{align} 
    M_t(\{i\}|i_0) := \frac{1-\exp(-\lambda N t)}{N} &+ \exp(-\lambda N t) \mathbf{1}[i = i_0],
\end{align} for $i, i_0 \in I, t \geq 0.$ Then, $$M_t(\cdot | i_0) = \mathbb{P}(\bsi(t) \in \cdot| \bsi(0) = i_0) \qquad (i_0 \in I, t \geq 0).$$
Moreover, let $M_{t|t_0}': I \times 2^I \rightarrow[0,1]$ be given by \begin{align} 
M_{t|t_0}'(\{j\}|j_0) := \frac{1-\exp\left(-N \int_{t_0}^t \mu(u) \mathrm{d}u\right)}{N} + \exp\left(- N \int_{t_0}^t \mu(u) \mathrm{d}u\right) \mathbf{1}[j = j_0], 
\end{align}
for $j, j_0 \in I$ and $ t \geq t_0 \geq 0.$
Then,
$$M_{t|t_0}'(\cdot | j_0) = \mathbb{P}(\bsj(t) \in \cdot| \bsj(t_0) = j_0) \qquad\qquad (j_0 \in I, t \geq t_0 \geq 0).$$
\end{lemma}
\begin{proof}
We prove only the assertion concerning $(\bsj(t))_{t \geq 0}$, the proof for $(\bsi(t))_{t \geq 0}$ is analogous.
Indeed, we show that $(M_{t|t_0}'(\{j\}|j_0))_{j, j_0 \in I}$
satisfies the Kolmogorov forward equation for any $t_0 \geq 0$:
\begin{align} \label{Kolmogorov_equation}
\frac{\partial M_{t|t_0}'(\{j\}|j_0)}{\partial t} &= \sum_{k=1}^N B(t)_{k,j}M_{t|t_0}'(\{k\}|j_0)  \qquad \qquad  (j \in I, t \geq t_0),\\
(M_{t_0|t_0}'(\{j\}|j_0))_{j_0,j \in I} &= \mathrm{Id}_I. \label{Kolmogorov_equation_initial}
\end{align}
For details, we refer to the fundamental work by Kolmogorov \cite[Equations (47), (52)]{Kolmogorov1931}.
The initial condition \eqref{Kolmogorov_equation_initial} is obviously satisfied. Moving on to \eqref{Kolmogorov_equation}.
We have 
\begin{align*}
&\frac{\partial M_{t|t_0}'(\{j\}|j_0)}{\partial t} =\mu(t)\exp\left(-N \int_{t_0}^t \mu(u) \mathrm{d}u\right) -N \mu(t) \exp\left(- N \int_{t_0}^t \mu(u) \mathrm{d}u\right) \mathbf{1}[j = j_0].
\end{align*}
Due to symmetry, it is sufficient to consider the cases  $j=j_0$ and $j \neq j_0$. Let first $j=j_0$. 
Then,
\begin{align*}
\frac{\partial M_{t|t_0}'(\{j\}|j_0)}{\partial t} &=(1-N)\mu(t)\exp\left(-N \int_{t_0}^t \mu(u) \mathrm{d}u\right) 
\\ &= \left(\frac{1-N}{N}\right)\mu(t)\left(1- (1-N) \exp\left(- N \int_{t_0}^t \mu(u) \mathrm{d}u\right)  \right) \\ &\quad-\left(\frac{1-N}{N}\right)\mu(t)\left(1- \exp\left(- N \int_{t_0}^t \mu(u) \mathrm{d}u\right)\right)
\\ &= B(t)_{j_0,j_0}M_{t|t_0}'(\{j_0\}|j_0)   +  \sum_{k=1, k \neq j_0}^N B(t)_{k,j}M_{t|t_0}'(\{k\}|j_0) 
\\ &= \sum_{k=1}^N B(t)_{k,j}M_{t|t_0}'(\{k\}|j_0).
\end{align*}
If on the other hand, $j \neq j_0$, we have
\begin{align*}
\frac{\partial M_{t|t_0}'(\{j\}|j_0)}{\partial t} &=\mu(t)\exp\left(-N \int_{t_0}^t \mu(u) \mathrm{d}u\right) 
\\ &= \mu(t) \left( M_{t|t_0}'(\{j_0\}|j_0) - M_{t|t_0}'(\{j\}|j_0) \right)
\\ &= \mu(t) \Big( M_{t|t_0}'(\{j_0\}|j_0) - (N-1) M_{t|t_0}'(\{j\}|j_0) + (N-2) M_{t|t_0}'(\{j\}|j_0)\Big)
\\ &= B(t)_{j_0,j}M_{t|t_0}'(\{j_0\}|j_0) + B(t)_{j,j}M_{t|t_0}'(\{j\}|j_0)   +  \sum_{k=1, k \neq j_0,j}^N B(t)_{k,j}M_{t|t_0}(\{k\}|j_0) 
\\ &= \sum_{k=1}^N B(t)_{k,j}M_{t|t_0}'(\{k\}|j_0).
\end{align*}
Hence, $M_{t|t_0}'$ is indeed  the Markov kernel describing the transition of the CTMP $(\bsj(t))_{t \geq 0}$.
\end{proof}
We now move on to proving the non-explosiveness of $(\bsj(t))_{t \geq 0}$.

\begin{lemma}\label{Lemma_non-explosive}
Let $(T_k)_{k=1}^\infty$ be the jump times of $(\bsj(t))_{t \geq 0}$. Then, $$\mathbb{P}\left(\lim_{k \rightarrow \infty} T_k = \infty\right) = 1.$$
\end{lemma}
\begin{proof}
In the following, we construct a CTMP $(\bsk(t))_{t \geq 0}$ on $\mathbb{N}$ which has the same jump times $(T_k)_{k=0}^\infty$ as $(\bsj(t))_{t \geq 0}$. 
Then, we show that $(\bsk(t))_{t \geq 0}$ satisfies the assumptions of Proposition 1 in \cite{Chow2011} on any compact intervall in $[0,\infty)$. This will imply our assertion.
Let $(\bsk(t))_{t \geq 0}$ be the CTMP on $\mathbb{N}$ with transition rate matrix
$
\Lambda(t): \mathbb{R}^{\mathbb{N}} \rightarrow  \mathbb{R}^{\mathbb{N}} (t \geq 0),
$
with $$\Lambda(t)_{i,j} = \begin{cases} (1-N)\mu(t), &\text{if }j = i\\  (N-1)\mu(t), &\text{if }j =i+1, \\  0, &\text{otherwise.} \end{cases}$$
We now need to check the following assertions for $\overline{t}>\underline{t}\geq 0.$
\begin{itemize}
\item[(i)] $\inf_{t \in [\underline{t}, \overline{t}]} \sum_{i=1}^n\frac{1}{-\Lambda(t)_{i,i}} \rightarrow \infty$, as $n \rightarrow \infty$,
\item[(ii)] there is a constant $C_{\overline{t}}^{(0)}>0$, such that 
$-\Lambda(t)_{i,i} > C_{\overline{t}}^{(0)}(-\Lambda(t)_{j,j}) $, for $i >j$, $t \in  [\underline{t}, \overline{t}]$
\item[(iii)] there is a constant $C_{\overline{t}}^{(1)}>0$, such that 
$$
\sum_{j=i+1}^\infty\frac{\Lambda(t)_{i,j}}{-\Lambda(t)_{i,i}}(j-i) \leq C_{\overline{t}}^{(1)} \sum_{j=1}^i \frac{1}{-\Lambda(t)_{j,j}},
$$
for $i \in \mathbb{N}, t \in [\underline{t}, \overline{t}]$,
and 
$$
\left\lvert-\frac{\partial}{\partial t}{\Lambda}(t)_{i,i} \right\rvert \leq C_{\overline{t}}^{(1)}(-\Lambda(t)_{i,i}).
$$
for $i \in \mathbb{N}, t \in [\underline{t}, \overline{t}]$.
\end{itemize}
Since $-\Lambda(t)_{i,i}$ is constant in $i \in \mathbb{N}$ and non-decreasing in $t$, the infimum in (i) is given by $-n/\Lambda(\overline{t})_{1,1}.$ This indeed goes to $\infty$, as $n \rightarrow \infty$, proving (i).
Again, as $-\Lambda(t)_{i,i}$ is constant in $i \in \mathbb{N}$, (ii) holds, with $C_{\overline{t}}^{(0)} = 0.9$.
Moreover, we have
$$
\sum_{j=i+1}^\infty\frac{\Lambda(t)_{i,j}}{-\Lambda(t)_{i,i}}(j-i)  = 1.
$$
Choosing $C_{\overline{t}}^{(1)} \geq -\Lambda(\overline{t})_{1,1}$, we can verify the first assertion of (iii), since $-\Lambda(t)_{1,1}$ is non-decreasing in $t$. The second assertion of (iii) is implied by \eqref{Eq_bounded_deriv}. (iii) is satisfied with $C_{\overline{t}}^{(1)}:= \max \{-\Lambda(\overline{t})_{1,1}, C_{\overline{t}}\}$.
\end{proof}
}

\bibliographystyle{plain}      
\bibliography{library}   

\end{document}

%% file: figure_PDMP.tex
\begin{figure*}
    \centering
  \begin{tikzpicture}[scale=1.1]
    \coordinate (y) at (0,2);
    \coordinate (x) at (12,0);
    \draw[thick,->] (y) -- (0,0) --  (x) node[right]
    {$t$};
    \draw[-,thick] (-0,2) node[left] {$N$}  -- (0.1,2)
    (-0,1.5) node[left] {$N-1$}  -- (0.1,1.5)
    (-0.2,1.1) node[left] {$\vdots$}
    (-0,0.5) node[left] {$2$} -- (0.1,0.5)
    (-0,0)  node[left] {$1$} -- (0.1,0)
    (0,-3.5) node[below] {$T_0$} -- (0,-3.4)
        (1.5,-3.5) node[below] {$T_1$} -- (1.5,-3.4)

    (2.,-3.5) node[below] {$T_2$} -- (2.,-3.4)
    (2.7,-3.5) node[below] {$T_3$} -- (2.7,-3.4)
    (4.5,-3.5) node[below] {$T_4$} -- (4.5,-3.4)
    (6.9,-3.5) node[below] {$T_5$} -- (6.9,-3.4)
    (7.4,-3.5) node[below] {$T_6$} -- (7.4,-3.4) 
        (9,-3.5) node[below] {$T_7$} -- (9,-3.4)
    (10.2,-3.5) node[below] {$T_8$\quad } -- (10.2,-3.4)
    (10.55,-3.5) node[below] {\quad $T_9$} -- (10.55,-3.4)
    (11.7,-3.5) node[below] {$T_{10}$} -- (11.7,-3.4) 
    (11.2,1.85) node[below] {$\bsi(t)$}
    (11.2,-1.15) node[below] {$\theta(t)$};
    \draw[dotted,thick] (1.5, -3.4) -- (1.5, 2)
    (2, -3.4) -- (2, 2)
    (2.7, -3.4) -- (2.7, 2)
    (4.5, -3.4) -- (4.5, 2)
    (6.9, -3.4) -- (6.9, 2)
    (7.4, -3.4) -- (7.4, 2)
    (9, -3.4) -- (9, 2)
    (10.2, -3.4) -- (10.2, 2)
    (10.55, -3.4) -- (10.55, 2)
    (11.7, -3.4) -- (11.7, 2);
    \coordinate (y) at (0,-1);
    \coordinate (x) at (12,-3.5);
    \draw[thick,<->] (y) node[left] {$X$} -- (0,-3.5)  --  (x) node[right] {$t$};

    \draw[{Circle}-{Circle[open]}] (0-0.065,2) -- (1.5+0.065,2);
    \draw[{Circle}-{Circle[open]}] (1.5-0.065,1) -- (2+0.065,1);
    \draw[{Circle}-{Circle[open]}] (2-0.065,1.5) -- (2.7+0.065,1.5);
    \draw[{Circle}-{Circle[open]}] (2.7-0.065,0.5) -- (4.5+0.065,0.5);
    \draw[{Circle}-{Circle[open]}] (4.5-0.065,1.5) -- (6.9+0.065,1.5);
    \draw[{Circle}-{Circle[open]}] (6.9-0.065,2) -- (7.4+0.065, 2);
    \draw[{Circle}-{Circle[open]}] (7.4-0.065,0) -- (9+0.0655, 0);
    \draw[{Circle}-{Circle[open]}] (9-0.065,2) -- (10.2+0.065, 2);
    \draw[{Circle}-{Circle[open]}] (10.2-0.065,1) -- (10.55+0.065,1);
    \draw[{Circle}-{Circle[open]}] (10.55-0.065,0) -- (11.7+0.065,0);
    
    \draw[thin] plot [smooth,tension=0] coordinates {(0,-2) (0,-2) (0.016746,-2.0972) (0.033492,-2.1881) (0.050238,-2.2731) (0.066984,-2.3526) (0.10448,-2.5124) (0.14198,-2.65) (0.17948,-2.7684) (0.21698,-2.8703) (0.25448,-2.958) (0.29198,-3.0335) (0.32948,-3.0985) (0.36698,-3.1544) (0.40448,-3.2025) (0.44198,-3.244) (0.47948,-3.2796) (0.51698,-3.3103) (0.55448,-3.3367) (0.59198,-3.3595) (0.62948,-3.3791) (0.66698,-3.3959) (0.70448,-3.4104) (0.74198,-3.4229) (0.77948,-3.4336) (0.81698,-3.4429) (0.85448,-3.4508) (0.89198,-3.4577) (0.92948,-3.4636) (0.96698,-3.4686) (1.0045,-3.473) (1.042,-3.4768) (1.0795,-3.48) (1.117,-3.4828) (1.1545,-3.4852) (1.192,-3.4873) (1.2295,-3.489) (1.267,-3.4906) (1.3045,-3.4919) (1.342,-3.493) (1.3795,-3.494) (1.417,-3.4948) (1.4377,-3.4952) (1.4585,-3.4956) (1.4792,-3.496) (1.5,-3.4963) (1.5,-3.4963) (1.5125,-3.3745) (1.525,-3.2587) (1.5375,-3.1486) (1.55,-3.0438) (1.5625,-2.9441) (1.575,-2.8493) (1.5875,-2.7591) (1.6,-2.6733) (1.6125,-2.5917) (1.625,-2.5141) (1.6375,-2.4402) (1.65,-2.37) (1.6625,-2.3032) (1.675,-2.2396) (1.6875,-2.1792) (1.7,-2.1217) (1.7125,-2.0669) (1.725,-2.0149) (1.7375,-1.9654) (1.75,-1.9183) (1.7625,-1.8735) (1.775,-1.8309) (1.7875,-1.7904) (1.8,-1.7519) (1.8125,-1.7152) (1.825,-1.6803) (1.8375,-1.6471) (1.85,-1.6156) (1.8625,-1.5856) (1.875,-1.557) (1.8875,-1.5298) (1.9,-1.504) (1.9125,-1.4794) (1.925,-1.456) (1.9375,-1.4338) (1.95,-1.4126) (1.9625,-1.3925) (1.975,-1.3734) (1.9875,-1.3552) (2,-1.3378) (2,-1.3378) (2.0145,-1.4031) (2.0289,-1.4648) (2.0434,-1.5229) (2.0578,-1.5778) (2.0753,-1.6402) (2.0928,-1.6983) (2.1103,-1.7525) (2.1278,-1.803) (2.1453,-1.8502) (2.1628,-1.8941) (2.1803,-1.9351) (2.1978,-1.9733) (2.2153,-2.0089) (2.2328,-2.0421) (2.2503,-2.073) (2.2678,-2.1019) (2.2853,-2.1288) (2.3028,-2.1539) (2.3203,-2.1773) (2.3378,-2.1991) (2.3553,-2.2195) (2.3728,-2.2384) (2.3903,-2.2561) (2.4078,-2.2726) (2.4253,-2.288) (2.4428,-2.3023) (2.4603,-2.3157) (2.4778,-2.3281) (2.4953,-2.3398) (2.5128,-2.3506) (2.5303,-2.3607) (2.5478,-2.3701) (2.5653,-2.3789) (2.5828,-2.3871) (2.6003,-2.3947) (2.6178,-2.4018) (2.6353,-2.4085) (2.6528,-2.4147) (2.6703,-2.4204) (2.6878,-2.4258) (2.6909,-2.4267) (2.6939,-2.4276) (2.697,-2.4285) (2.7,-2.4293) (2.7,-2.4293) (2.7449,-2.3175) (2.7898,-2.2242) (2.8347,-2.1463) (2.8797,-2.0812) (2.9247,-2.0266) (2.9697,-1.981) (3.0147,-1.943) (3.0597,-1.9112) (3.1047,-1.8846) (3.1497,-1.8624) (3.1947,-1.8439) (3.2397,-1.8285) (3.2847,-1.8155) (3.3297,-1.8047) (3.3747,-1.7957) (3.4197,-1.7882) (3.4647,-1.7819) (3.5097,-1.7766) (3.5547,-1.7723) (3.5997,-1.7686) (3.6447,-1.7655) (3.6897,-1.763) (3.7347,-1.7608) (3.7797,-1.7591) (3.8247,-1.7576) (3.8697,-1.7563) (3.9147,-1.7553) (3.9597,-1.7544) (4.0047,-1.7537) (4.0497,-1.7531) (4.0947,-1.7526) (4.1397,-1.7521) (4.1847,-1.7518) (4.2297,-1.7515) (4.2747,-1.7513) (4.3197,-1.751) (4.3647,-1.7509) (4.4098,-1.7507) (4.4549,-1.7506) (4.5,-1.7505) (4.5,-1.7505) (4.5293,-1.8335) (4.5587,-1.9073) (4.588,-1.9729) (4.6173,-2.0313) (4.6773,-2.1315) (4.7373,-2.2103) (4.7973,-2.2719) (4.8573,-2.3204) (4.9173,-2.3588) (4.9773,-2.389) (5.0373,-2.4126) (5.0973,-2.4312) (5.1573,-2.4459) (5.2173,-2.4575) (5.2773,-2.4665) (5.3373,-2.4736) (5.3973,-2.4793) (5.4573,-2.4837) (5.5173,-2.4872) (5.5773,-2.4899) (5.6373,-2.4921) (5.6973,-2.4938) (5.7573,-2.4951) (5.8173,-2.4961) (5.8773,-2.497) (5.9373,-2.4976) (5.9973,-2.4981) (6.0573,-2.4985) (6.1173,-2.4988) (6.1773,-2.4991) (6.2373,-2.4993) (6.2973,-2.4994) (6.3573,-2.4996) (6.4173,-2.4996) (6.4773,-2.4997) (6.5373,-2.4998) (6.5973,-2.4998) (6.6573,-2.4999) (6.7173,-2.4999) (6.7773,-2.4999) (6.808,-2.4999) (6.8387,-2.4999) (6.8693,-2.4999) (6.9,-2.4999) (6.9,-2.4999) (6.9125,-2.5487) (6.925,-2.5951) (6.9375,-2.6392) (6.95,-2.6812) (6.9625,-2.7212) (6.975,-2.7591) (6.9875,-2.7953) (7,-2.8296) (7.0125,-2.8623) (7.025,-2.8934) (7.0375,-2.923) (7.05,-2.9512) (7.0625,-2.9779) (7.075,-3.0034) (7.0875,-3.0276) (7.1,-3.0506) (7.1125,-3.0726) (7.125,-3.0934) (7.1375,-3.1132) (7.15,-3.1321) (7.1625,-3.15) (7.175,-3.1671) (7.1875,-3.1833) (7.2,-3.1988) (7.2125,-3.2135) (7.225,-3.2275) (7.2375,-3.2407) (7.25,-3.2534) (7.2625,-3.2654) (7.275,-3.2769) (7.2875,-3.2877) (7.3,-3.2981) (7.3125,-3.3079) (7.325,-3.3173) (7.3375,-3.3262) (7.35,-3.3347) (7.3625,-3.3428) (7.375,-3.3504) (7.3875,-3.3577) (7.4,-3.3647) (7.4,-3.3647) (7.44,-3.3107) (7.48,-3.2648) (7.52,-3.2256) (7.56,-3.1923) (7.6,-3.1638) (7.64,-3.1396) (7.68,-3.119) (7.72,-3.1014) (7.76,-3.0864) (7.8,-3.0736) (7.84,-3.0627) (7.88,-3.0535) (7.92,-3.0456) (7.96,-3.0388) (8,-3.0331) (8.04,-3.0282) (8.08,-3.024) (8.12,-3.0205) (8.16,-3.0174) (8.2,-3.0149) (8.24,-3.0127) (8.28,-3.0108) (8.32,-3.0092) (8.36,-3.0078) (8.4,-3.0067) (8.44,-3.0057) (8.48,-3.0049) (8.52,-3.0041) (8.56,-3.0035) (8.6,-3.003) (8.64,-3.0026) (8.68,-3.0022) (8.72,-3.0019) (8.76,-3.0016) (8.8,-3.0013) (8.84,-3.0011) (8.88,-3.001) (8.92,-3.0008) (8.96,-3.0007) (9,-3.0006) (9,-3.0006) (9.03,-3.0571) (9.06,-3.1072) (9.09,-3.1516) (9.12,-3.191) (9.15,-3.2259) (9.18,-3.2569) (9.21,-3.2844) (9.24,-3.3088) (9.27,-3.3304) (9.3,-3.3496) (9.33,-3.3666) (9.36,-3.3817) (9.39,-3.3951) (9.42,-3.4069) (9.45,-3.4174) (9.48,-3.4268) (9.51,-3.4351) (9.54,-3.4424) (9.57,-3.4489) (9.6,-3.4547) (9.63,-3.4598) (9.66,-3.4644) (9.69,-3.4684) (9.72,-3.472) (9.75,-3.4751) (9.78,-3.4779) (9.81,-3.4804) (9.84,-3.4827) (9.87,-3.4846) (9.9,-3.4864) (9.93,-3.4879) (9.96,-3.4893) (9.99,-3.4905) (10.02,-3.4916) (10.05,-3.4925) (10.08,-3.4934) (10.11,-3.4941) (10.14,-3.4948) (10.17,-3.4954) (10.2,-3.4959) (10.2,-3.4959) (10.2087,-3.41) (10.2175,-3.3272) (10.2263,-3.2471) (10.235,-3.1698) (10.2437,-3.0952) (10.2525,-3.0231) (10.2613,-2.9535) (10.27,-2.8864) (10.2787,-2.8215) (10.2875,-2.7588) (10.2963,-2.6983) (10.305,-2.6399) (10.3137,-2.5835) (10.3225,-2.529) (10.3313,-2.4765) (10.34,-2.4257) (10.3487,-2.3766) (10.3575,-2.3293) (10.3663,-2.2836) (10.375,-2.2394) (10.3837,-2.1968) (10.3925,-2.1556) (10.4013,-2.1159) (10.41,-2.0775) (10.4187,-2.0404) (10.4275,-2.0047) (10.4363,-1.9701) (10.445,-1.9367) (10.4537,-1.9045) (10.4625,-1.8734) (10.4713,-1.8434) (10.48,-1.8144) (10.4887,-1.7863) (10.4975,-1.7593) (10.5063,-1.7332) (10.515,-1.708) (10.5237,-1.6836) (10.5325,-1.6601) (10.5413,-1.6374) (10.55,-1.6155) (10.55,-1.6155) (10.5647,-1.6943) (10.5793,-1.7686) (10.594,-1.8387) (10.6086,-1.9049) (10.6374,-2.0238) (10.6661,-2.1299) (10.6949,-2.2244) (10.7236,-2.3087) (10.7524,-2.3838) (10.7811,-2.4507) (10.8099,-2.5104) (10.8386,-2.5636) (10.8674,-2.611) (10.8961,-2.6532) (10.9249,-2.6909) (10.9536,-2.7245) (10.9824,-2.7544) (11.0111,-2.7811) (11.0399,-2.8049) (11.0686,-2.8261) (11.0974,-2.845) (11.1261,-2.8618) (11.1549,-2.8768) (11.1836,-2.8902) (11.2124,-2.9021) (11.2411,-2.9128) (11.2699,-2.9222) (11.2986,-2.9307) (11.3274,-2.9382) (11.3561,-2.9449) (11.3849,-2.9509) (11.4136,-2.9562) (11.4424,-2.961) (11.4711,-2.9652) (11.4999,-2.969) (11.5286,-2.9724) (11.5574,-2.9754) (11.5861,-2.9781) (11.6149,-2.9804) (11.6436,-2.9826) (11.6577,-2.9835) (11.6718,-2.9844) (11.6859,-2.9853) (11.7,-2.9861)  };
    \end{tikzpicture}
    \caption{Cartoon of SGPC: the process $(\bsi(t))_{t \geq 0}$ is a right continuous, piecewise constant process on the set $I$, whereas the process $(\theta(t))_{t \geq 0}$ on $X$ is continuous and piecewise smooth. The pieces on which the processes are constant resp. smooth are identical, since the dynamic of $(\theta(t))_{t \geq 0}$ is controlled by $(\bsi(t))_{t \geq 0}$. Note that, $T_0$ is the initial time and the increments $T_k -T_{k-1}$ are the random waiting times.}
    \label{Fig_pdmp_cartoon}
\end{figure*}